\documentclass{amsart}[12pt]
\usepackage{amsmath, amsthm, amscd, amsfonts,}
\usepackage{graphicx,float}

\usepackage{amssymb}
\usepackage{pb-diagram}
\usepackage{lamsarrow}
\usepackage{pb-lams}
\usepackage{comment}

\DeclareMathOperator{\crit}{crit}
\DeclareMathOperator{\dom}{dom}

\DeclareMathOperator{\Col}{Col}
\DeclareMathOperator{\Add}{Add}

\DeclareMathOperator{\Ult}{Ult}

\def\MPB{{\mathbb{P}}}
\def\MQB{{\mathbb{Q}}}
\def\MRB{{\mathbb{R}}}
\def\MCB{{\mathbb{C}}}

\newtheorem{theorem}{Theorem}[section]
\newtheorem{lemma}[theorem]{Lemma}

\newtheorem{definition}[theorem]{Definition}

\newtheorem{remark}[theorem]{Remark}
\newtheorem{claim}[theorem]{Claim}

\newtheorem{notation}[theorem]{Notation}
\numberwithin{equation}{section}

\newcommand{\NSP}{{\rm NSP}}

\usepackage{pb-diagram}

\def\rmark{\mbox{$\rm\bf\rule{0.06em}{1.45ex}\kern-0.05em R$}}
\def\pmark{\mbox{$\rm\bf\rule{0.06em}{1.45ex}\kern-0.05em P$}}
\def\nmark{\mbox{$\rm\bf\rule{0.06em}{1.45ex}\kern-0.05em N$}}
\def\vdash{\mbox{$\rm\| \kern-0.13em -$}}

\def\rmark{\mbox{$\rm\bf\rule{0.06em}{1.45ex}\kern-0.05em R$}}
\def\pmark{\mbox{$\rm\bf\rule{0.06em}{1.45ex}\kern-0.05em P$}}
\def\nmark{\mbox{$\rm\bf\rule{0.06em}{1.45ex}\kern-0.05em N$}}
\def\vdash{\mbox{$\rm\| \kern-0.13em -$}}
\newcommand{\lusim}[1]{\smash{\underset{\raisebox{1.2pt}[0cm][0cm]{$\sim$}}
{{#1}}}}

\begin{document}

\title[Tree property at successor of a singular limit of  measurable cardinals]{Tree property at successor of a singular limit of  measurable cardinals}

\author[M. Golshani ]{Mohammad Golshani}

  \thanks{The author's research was in part supported by a grant from IPM (No. 91030417). }

\maketitle

\begin{abstract}
Assume $\lambda$ is a singular limit of $\eta$ supercompact cardinals, where $\eta \leq \lambda$ is a limit ordinal. We present two forcing methods for making
$\lambda^+$  the successor of the limit of the first $\eta$ measurable cardinals while the tree property holding at $\lambda^+.$ The first method is then used to get,
from the same assumptions,
tree property at $\aleph_{\eta^2+1}$ with the failure of $SCH$ at $\aleph_{\eta^2}$. This extends  results of Neeman and Sinapova. The second method is also used to get tree property at successor of an arbitrary singular cardinal, which extends some results of Magidor-Shelah, Neeman and Sinapova.
\end{abstract}
\section{introduction}
An important theorem of Magidor \cite{magidor} says that it is consistent, relative to the existence of a strongly compact cardinal, that the least strongly compact cardinal is also the least measurable cardinal. This result is  generalized by Kimchi-Magidor (see \cite{grigor}) who showed that for any natural number $n,$ it is consistent, relative to the existence of $n$ supercompact cardinals, that there are at least $n$ strongly compact cardinals, and  the first $n$ strongly compact cardinals are exactly the first $n$ measurable cardinals.
A major open problem  is  if it is consistent for the first $\omega$ strongly compact cardinals to coincide with the first $\omega$ measurable cardinals.
In \cite{grigor}, Sargsyan, describes the problems as follows: ``It is a difficult problem, one whose
ultimate solution might just lie elsewhere then the places that were suspected in the
past. Understanding the combinatorics of $\lambda^+$ where $\lambda$ is a limit of strongly compact
cardinals might eventually lead to its negative resolution''.

On the other hand a result of Magidor-Shelah \cite{magidor-shelah} says that if $\lambda$ is a singular limit of strongly compact cardinals, then $\lambda^+$ has the tree property.
However this result is not necessarily true for measurable cardinals. This can be seen either using results from core model theory or by a simple forcing argument.
Assume $V= \mathcal{K},$ where $\mathcal{K}$ is the core model for a strong cardinal, and suppose that $\lambda$ is a singular limit of measurable cardinals.
Then $\Box^*_\lambda$ holds, in particular there are special $\lambda^+$-Aronszajn trees in $\mathcal{K},$ hence tree property fails at
$\lambda^+.$ On the other hand, if $\lambda$ is as above and if we force  $\Box^*_\lambda$ using its initial approximations, then in the resulting generic extension,
$\lambda$ remains the singular limit of measurable cardinals, but tree property fails at
$\lambda^+.$

Motivated by these results, we prove  the following.
\begin{theorem}
Assume $\lambda$ is a singular limit of $\eta$ supercompact cardinals, where $\eta \leq \lambda$ is a limit ordinal. Then there is a generic extension in which:

$(a)$ $\lambda^+$ is preserved and it is  the successor of the limit of the first $\eta$ measurable cardinals,

$(b)$ Tree property holds at $\lambda^+.$
\end{theorem}
Our result shows that the tree property is not a suitable candidate for the study of the above cited problem, as suggested by Sargsyan.
It is worth mentioning that it is  easy to modify  Neeman's proof of tree property at $\aleph_{\omega+1}$ in  \cite{neeman} to get the
tree property at the successor of the supremum of the first $\omega$ measurable cardinals. This is enough to show that the tree property is not  sufficiently
strong to separate the first limit of measurable cardinals from the first limit of strongly compact
cardinals. The main advantage of the above theorem is that it works for all cofinalities
and also the proofs of the theorem allow us to get more results about tree property.

In this paper we will present two different proofs of the above theorem, the fist one uses diagonal Magidor forcing with interleaved collapses
and the second one uses Levy collapses.

The diagonal Magidor forcing with interleaved collapses was introduced by Sinapova \cite{sinapova} (see also \cite{sinapova2}), where she used the forcing to get
the failure of $SCH$ at $\aleph_{\omega_1^2}$ together with the existence of a very good scale and a bad scale at $\aleph_{\omega_1^2+1}$.
Our fist proof of Theorem 1.1 uses the above forcing notion and is much more complicated than the second proof which uses Levy collapses.
The reason for giving such a proof is that, unlike the second proof, the method is more flexible and allows us to add the failure of $SCH$
into our conclusion. To be more precise, we uss the method
 to prove the following theorem, which extends  results of Neeman \cite{neeman0}
and Sinapova \cite{sinapova2}, \cite{sinapova3}.

\begin{theorem}
Assume $\lambda$ is a singular limit of $\eta$ supercompact cardinals, where $\eta \leq \lambda$ is a limit ordinal. Then there is a generic extension in which:

$(a)$ $\lambda^+=\aleph_{\eta^2+1}$,

$(b)$ $\aleph_{\eta^2}$ is a singular strong limit cardinals and $2^{\aleph_{\eta^2}} > \aleph_{\eta^2+1},$

$(c)$ Tree property holds at $\aleph_{\eta^2+1}.$

\end{theorem}
The second method of the proof uses the product of Levy collapses over a suitable prepared model. This method was first used by Neeman  \cite{neeman}
 to get tree property at $\aleph_{\omega+1}.$ We extend Neeman's proof to cover uncountable cofinalities and use it to present a different proof
 of Theorem 1.1.
 The method of the proof is also used to get the following theorem, which extends the results of
Magidor-Shelah \cite{magidor-shelah}, Neeman \cite{neeman} and Sinapova \cite{sinapova1}.
\begin{theorem}
Assume $\lambda$ is a singular limit of $\eta$ supercompact cardinals, where $\eta \leq \lambda$ is a limit ordinal. Then there is a generic extension in which
 $\lambda^+=\aleph_{\eta+1}$ and tree property holds at $\aleph_{\eta+1}.$
\end{theorem}

The structure of the paper is as follows. Section 2 is devoted to some  preservation lemmas that will be used later.
In section 3, we give a proof of  theorem 1.1 using diagonal Magidor forcing with
 interleaved collapses, developed by
  Sinapova \cite{sinapova}, \cite{sinapova2}.
  In section 4, we show  how to modify the above argument to get a proof of Theorem 1.2.
In section 5, we present another proof of Theorem 1.1 which uses ideas developed by Nemman \cite{neeman} and finally in section 6, we sketch the proof
 of Theorem 1.3.

\section{Some preservation lemmas}
In this section we present some definitions and results that will be used in later sections of this paper.
Our main tool for the study of tree property is the notion of (narrow) systems introduced by Magidor and Shelah \cite{magidor-shelah}.
\begin{definition}
Let $\kappa$ be a regular cardinal. A narrow system at $\kappa$ is a tuple $\mathcal{S} = \langle I, \mathcal{R}\rangle$ where for some $\rho$, we have:
\begin{enumerate}
\item (Narrowness) $\rho^+ < \kappa$, $|\mathcal{R}|^+ < \kappa$ and  $I$ is unbounded in $\kappa$.
\item For every $R\in\mathcal{R}$, $R$ is a tree like partial order on $I\times\rho$, i.e.\ it is a transitive, reflective and anti-symmetric relation and for every $x, y, z$ if $y R x$ and $z R x$ then either $y R z$ or $z R y$.
\item For every $R\in\mathcal{R}$, $\langle \alpha, \zeta\rangle R \langle \beta, \xi\rangle$ implies  $\alpha \leq \beta$ and if $\alpha = \beta$ then $\zeta = \xi$.
\item (Connectedness) for every $\alpha < \beta$ in $I$ there are $R\in\mathcal{R}$, $\zeta, \xi < \rho$ such that $\langle \alpha, \zeta\rangle R \langle \beta, \xi\rangle$.
\end{enumerate}
Sets of the form $\{\alpha\}\times\rho$ are called the levels of $\mathcal{S}$.

A branch through a narrow system $\mathcal{S}$ as above is a subset $b\subseteq I\times \rho$ such that there is $R\in\mathcal{R}$ for which $\langle b, R\rangle$ is a total order. $b$ is cofinal if it meets cofinally many levels in $\mathcal{S}$.

A system of branches is a collection of branches $\{b_j \mid j < |\mathcal{R}\times \rho|\}$ such that each $b_j$ is a branch and for every $\alpha\in I$ there is $j$ such that $b_j\cap (\{\alpha\}\times\rho) \neq \emptyset$.
\end{definition}

\begin{definition}
Let $\kappa$ be a regular cardinal. The narrow system property at $\kappa$, $\NSP(\kappa)$ is the assertion that every narrow system at $\kappa$ has a cofinal branch.
\end{definition}
Systems are tuples that satisfy all requirements except for the first from definition of narrow systems. Systems, which were first defined by Magidor and Shelah \cite{magidor-shelah}, appear naturally when analyzing names of trees under forcing notions.  The advantage of this definition is that even when restricting the system to a cofinal subset of $I$, it remains a system. In a similar way, a name of a system is a system. For more details about systems and narrow systems see \cite{LambieHanson2015}.

The next lemma is proved by Sinapova  \cite{sinapova1} in the case $\nu$ has cofinality $\omega;$ but her proof can be modified for all singular cardinals.
\begin{lemma} (Preservation lemma 1)
Let $\nu$ be a singular cardinal  and let $\kappa \leq \eta < \nu$ be  regular cardinals. Let $\mathbb{Q}$ be a $\eta^+$-c.c. forcing notion and let $\mathbb{R}$ be a $\eta^+$-closed forcing notion. Assume that in $V^{\mathbb{Q}}$ there is a narrow system $\mathcal{S}$ of height $\nu^+$, levels of width $\eta$ and index set of cardinality $\kappa$. Assume that in $V^{\mathbb{Q}\times \mathbb{R}}$ there is a system of branches for $\mathcal{S}$.

Then, in $V^{\mathbb{Q}}$ there is a cofinal branch for $\mathcal{S}$.
\end{lemma}
For the proof of Theorem 1.2 we need the following lemma from \cite{sinapova3}. Again the lemma is stated and proved for the countable cofinality case, but its proof
can be modified to get the following stronger result.
\begin{lemma} (Preservation lemma 2)
Suppose that $\nu$ is a singular cardinal,
$\kappa, \tau < \nu$ are regular cardinals, and in $V , \MQB$ is $\kappa^+$-c.c notion of forcing
and $\mathbb{R}$ is a $\max(\kappa, \tau)^+$-closed notion of forcing. Let $E$ be $\MQB$-generic
over $V$ and let $F$ be $\mathbb{R}$-generic over $V[E]$. Suppose that $\mathcal{S}= \langle I, \mathcal{R}\rangle$
is a narrow system in $V[E]$ of height $\nu^+$, levels of size $\kappa$, and with
$\mathcal{R}= \langle  R_\sigma: \sigma < \tau   \rangle$. Suppose that in $V[E][F]$ there are (not necessarily
all unbounded) branches $\langle  b_{\sigma, \delta}: \sigma \in L, \delta < \kappa    \rangle$, such that:
\begin{enumerate}
\item  every $b_{\sigma, \delta}$ is a branch through $R_\sigma$, and for some $(\sigma, \delta) \in L \times \kappa,$
$b_{\sigma, \delta}$  is unbounded.

\item  for all $\alpha \in I$, there is  $(\sigma, \delta) \in L \times \kappa$ such that $S_\alpha \cap b_{\sigma, \delta} \neq \emptyset.$
\end{enumerate}
Then $\mathcal{S}$ has an unbounded branch in $V[E]$.
\end{lemma}
We will frequently use Lemmas 2.3 and 2.4, by defining a tuple $\mathcal{S} = \langle I, \mathcal{R}\rangle$
and claiming that it forms a narrow system which satisfies the assumptions of the above lemmas, without verifying our claim. In such cases,
either our claim is trivial or can be obtained by standard arguments, such as those given in
\cite{sinapova1}, \cite{sinapova2} or \cite{sinapova3}.

\section{Diagonal Magidor forcing with interleaved collapses}
In this section we give the proof of Theorem 1.1 using diagonal Magidor forcing with interleaved collapses, introduced by Sinapova \cite{sinapova}. Thus suppose that $\lambda$  is a singular limit of $\eta$ supercompact cardinals, where $\eta \leq \lambda$ is a limit ordinal. If $\eta=\lambda,$ then $\lambda$ itself is the limit of the first $\eta$ measurable cardinals, and  there is nothing to do. So we can assume that
$\eta<\lambda.$ Let $\langle  \kappa_\xi: \xi \leq \eta      \rangle$
be an increasing and continuous sequence of cardinals so that:
\begin{enumerate}
\item Each $\kappa_\xi,$ where $\xi=0$ or $\xi$ is a successor ordinal is a supercompact cardinal,
\item $\kappa_\eta=\lambda.$
\end{enumerate}
We further assume that
\begin{enumerate}
\item [(3)] $\eta<\kappa_0=\kappa.$
\end{enumerate}
At the end of the proof, we will show  how to remove this extra assumption. Set $I=\{\xi:$ $\xi=0$ or $\xi$ is a successor ordinal$ \}$. It follows that for each
$\xi \leq \eta,$ if $\xi\in I$, then $\kappa_\xi$ is a supercompact cardinal and if $\xi \notin I,$ then $\kappa_\xi$ is a singular cardinal.
We may further assume that
\begin{enumerate}
\item [(4)] For $\xi\in I, \kappa_\xi$ is Laver indestructible under $\kappa_\xi$-directed closed forcing notions,
\item [(5)] For all $\xi \leq \eta, 2^{\kappa_\xi}=\kappa_{\xi}^+.$
\end{enumerate}
\begin{notation}
For each cardinal $\alpha,$ let $\alpha_*$ denote the least measurable cardinal above $\alpha$, if it exists, and let $\alpha_*$ be undefined otherwise.
\end{notation}
Using the above notation, we have the following lemma, that we will use it extensively without any mention.
\begin{lemma}
Assume $\alpha > \eta$ and $\alpha_*$ is well-defined. Then $\alpha^{+\eta} < \alpha_*;$ in particular $\forall \xi \leq \eta, (\alpha^{+\xi})_*=\alpha_*.$
\end{lemma}

Let $\mathbb{C}$ be the Easton support iteration of $\Col(\kappa_i^{+}, < \kappa_{i+1}), i<\eta,$ and let $H$ be $\mathbb{C}$-generic over $V$. The following lemma can be proved easily:
\begin{lemma}
$(a)$ $\mathbb{C}$ is $\kappa$-directed closed, in particular $\kappa$ remains supercompact in $V[H].$

$(b)$ In $V[H]$, for all $\xi \leq \eta, \kappa_\xi=\kappa^{+\xi}$.
\end{lemma}
Work in $V[H]$.
We now define our main forcing construction. The forcing defined below is motivated from \cite{sinapova} and
\cite{sinapova2}, where we refer  to them for more details.
The following lemma can be proved as in \cite{sinapova}, Lemma 3.3 and \cite{sinapova2}, Proposition 2.
\begin{lemma}
In $V[H],$ there are sequences $\langle U_\xi: \xi<\eta     \rangle$, 
 and $\langle K_\xi: \xi<\eta  \rangle$ such that:
\begin{enumerate}
\item [(a)] $U_\xi$ is a normal measure on $P_\kappa(\kappa^{+\xi})$,

\item [(b)] $\kappa$ is $\kappa^{+\xi}$-supercompact in $N_\xi=Ult(V[H], U_\xi),$

\item [(c)] The sequence $\langle U_\xi: \xi<\eta     \rangle$  is Mitchell increasing, i.e., $\zeta<\xi<\eta \Rightarrow U_\zeta \in N_\xi,$


\item [(d)] $K_\xi$ is $\Col(\kappa_*^{++}, < j_{U_\xi}(\kappa))_{N_\xi}$-generic over $N_\xi$.
\end{enumerate}
\end{lemma}
For $\xi<\eta,$ let $X_\xi$ be the set of all $x \in P_\kappa(\kappa^{+\xi})$ such that
\begin{enumerate}
\item $\kappa_x=x \cap \kappa$ is an ordinal $> \eta,$

\item For all $\zeta \leq \xi, otp(x \cap \kappa^{+\zeta})=\kappa_x^{+\zeta}$, in particular $otp(x)=\kappa_x^{+\xi}$,

\item For all $\zeta \leq \xi, (\kappa_x^{+\zeta})^{<\kappa_x} \leq \kappa_x^{+\zeta+1}.$


\end{enumerate}
By standard reflection arguments each $X_\xi\in U_\xi.$
For $\zeta < \xi < \eta, x \in X_\xi$ and $Y \subseteq P_{\kappa_x}(x \cap \kappa^{+\zeta}),$ let
\[
\bar{Y} =\{ \{otp(x \cap \delta): \delta \in y  \}: y \in Y     \} \subseteq P_{\kappa_x}(otp(x \cap \kappa^{+\zeta}))=P_{\kappa_x}(\kappa_x^{+\zeta}).
\]
Since $U_\zeta \in N_\xi,$ there is a function $x \mapsto \bar{U}^\zeta_{\xi, x}$ such that $U_\zeta = [x \mapsto \bar{U}^\zeta_{\xi, x}]_{U_\xi}$
and each $\bar{U}^\zeta_{\xi, x}$ is a normal measure on $P_{\kappa_x}(\kappa_x^{+\zeta})=P_{\kappa_x}(otp(x \cap \kappa^{+\zeta}))$.
Lift this measure to a normal measure
$U^\zeta_{\xi, x}$ on $P_{\kappa_x}(x \cap \kappa^{+\zeta})$, so that we have
\[
\bar{U}^\zeta_{\xi, x}=\{\bar{Y} \subseteq  P_{\kappa_x}(\kappa_x^{+\zeta}): Y \in U^\zeta_{\xi,x}            \}.
\]
For $\xi < \eta,$ let
\[
B_\xi=\{ z \in X_\xi: \forall \zeta < \tau < \xi ~(\bar{U}^{\zeta}_{\xi,z} = [x \mapsto \bar{U}^{\zeta}_{\eta,x} ]_{U^{\eta}_{\xi,z}})                   \}.
\]
Then $B_\xi \in U_\xi.$ Also for $\zeta < \xi < \eta$ and $x \in B_\xi$ let $j^{\zeta}_{\xi,x}=j_{\bar{U}^{\zeta}_{\xi,x}}$, and assume
$K^{\zeta}_{\xi,x}$ is such that $K_\zeta = [x \mapsto K^{\zeta}_{\xi,x}]_{U_\xi}$. Then
$K^{\zeta}_{\xi,x}$ is $\Col((\kappa_x)_*^{++}, < j^{\zeta}_{\xi,x}(\kappa_x))_{N^{\zeta}_{\xi,x}}$-generic
over $N^{\zeta}_{\xi,x},$ where  $N^{\zeta}_{\xi,x}=Ult(V, \bar{U}^{\zeta}_{\xi,x})$.

We are now ready to define our main forcing notion. The forcing is essentially the same as the forcing construction in section 3.2 of \cite{sinapova}.
\begin{definition}
A condition in $\MPB$ is a tuple $p= \langle  g,f, H, F     \rangle$ where:
\begin{enumerate}
\item $\dom(g)$ is a finite subset of $\eta$ and $\dom(H)=\eta \setminus \dom(g),$

\item For each $\xi \in \dom(g), g(\xi) \in B_\xi,$

\item For $\zeta < \xi$ in $\dom(g), g(\zeta) \prec g(\xi),$ i.e., $g(\zeta) \subseteq g(\xi)$ and $otp(g(\zeta)) < \kappa_{g(\xi)}$,

\item If $\xi > \max(\dom(g)),$ then $H(\xi) \in U_\xi, H(\xi) \subseteq B_\xi,$

\item If $\xi \notin \dom(g)$ and $\xi < \max(\dom(g)),$ then setting $\zeta=\min(\dom(g) \setminus \xi),$ we have
$H(\xi) \in U^\xi_{\zeta, g(\zeta)}$,

\item If $\xi < \zeta, \xi \in \dom(g)$ and $\zeta \notin \dom(g),$ then for all $z \in H(\zeta), g(\xi) \prec z,$

\item $\dom(f)=\dom(g),$ and for $\xi \in \dom(f)$:
\begin{enumerate}
\item [(7-1)] If  $\xi < \max(\dom(f)),$ then setting $\zeta=\min(\dom(f)\setminus \xi),$
we have $f(\xi) \in \Col((\kappa_{g(\xi)})_*^{++}, < \kappa_{g(\zeta)})$,

\item [(7-2)] If $\xi = \max(\dom(f)),$ then $f(\xi) \in \Col((\kappa_{g(\xi)})_*^{++}, < \kappa),$
\end{enumerate}
\item $\dom(F)=\dom(H)$,

\item For $\xi \in \dom(F), F(\xi)$ is a function with domain $H(\xi)$, and for $y \in H(\xi)$
\begin{enumerate}
\item [(9-1)] If $\xi < \max(\dom(g))$, then setting $\zeta=\min(\dom(g)\setminus \xi)$ and $x=g(\zeta),$
we have $F(\xi)(y) \in \Col((\kappa_y)_*^{++}, < \kappa_x)$ and $[\bar{F}(\xi)]_{U^\xi_{\zeta,x}} \in K^\xi_{\zeta,x},$
where $\bar{F}(\xi)$ is defined on $P_{\kappa_x}(otp(x \cap \kappa_\xi))$ by  $\bar{y} \mapsto F(\xi)(y)$ (where $\bar{y}=\{ otp(x \cap \eta): \eta \in y  \}$),

\item [(9-2)] If $\xi > \max(\dom(g)),$ then $F(\xi)(y) \in \Col((\kappa_y)_*^{++}, <\kappa)$ and $[F(\xi)]_{U_\xi} \in K_\xi.$

\end{enumerate}
\end{enumerate}
\end{definition}
Given a condition $p \in \MPB,$ we denote it by
\[
p= \langle  g^p, f^p, H^p, F^p      \rangle.
\]
We call $\langle  g^p, f^p  \rangle$ the stem of $p$, and denote it by $stem(p)$.
The following definition will be used later.
\begin{definition}
The pair $\langle  g, f  \rangle$ is a $\MPB$-stem, if there exists $p\in \MPB$ such that $stem(p)=\langle  g, f  \rangle $
\end{definition}
We now define the order relation.  In fact we define two kind of order relations $\leq$ and $\leq^*$.
\begin{definition}
Let $p, q \in \MPB.$

$(a)$ $p \leq q$ ($p$ is an extension of $q$) iff
\begin{enumerate}
\item $g^p \supseteq g^q,$

\item $\xi \in \dom(g^p) \setminus \dom(g^q) \Rightarrow g^p(\xi) \in H^q(\xi),$

\item If $\xi \notin \dom(g^p),$ then $H^p(\xi) \subseteq H^q(\xi),$

\item If $\xi \in \dom(f^q)$ and $\xi < \max(\dom(f^p)),$ then setting $\zeta=\min(\dom(f^p \setminus \xi))$, we have
$f^p(\xi) \leq f^q(\xi) \upharpoonright \kappa_{g^p(\zeta)}$,

\item If $\xi=\max(\dom(f^q)),$ then $f^p(\xi) \leq f^q(\xi),$

\item If $\xi \in \dom(f^p) \setminus \dom(f^q),$ then setting $\zeta=\min(\dom(f^p \setminus \xi))$, we have
$f^p(\xi) \leq F^q(\xi)(g^p(\xi)) \upharpoonright \kappa_{g^p(\beta)}$,

\item If $\xi \in \dom(f^p) \setminus \dom(f^q)$ and $\xi=\max(\dom(f^p)),$ then $f^p(\xi) \leq F^q(\xi)(g^p(\xi)),$

\item If $\max(\dom(g^p)) < \xi,$ then for each $y \in H^p(\xi), F^p(\xi)(y) \leq F^q(\xi)(y),$

\item If $\xi \notin \dom(g^p)$ and $\xi < \max(\dom(g^p)),$ then setting $\zeta=\min(\dom(g^p)\setminus \xi),$ for each $y \in H^p(\xi)$
we have $F^p(\xi)(y) \leq F^q(\xi)(y) \upharpoonright \kappa_{g^p(\zeta)}$.
\end{enumerate}

$(b)$ $p \leq^* q$ ($p$ is a Prikry or a direct extension of $q$) iff $p \leq q$ and $\dom(g^p)=\dom(g^q).$
\end{definition}
Let's state the main properties of the forcing notion $\MPB.$ Let $G$ be $\MPB$-generic over $V[H]$, and let $g^*=\bigcup_{p\in G}g^p.$
\begin{lemma}
(Basic properties of the forcing notion $\MPB$)

\begin{enumerate}
\item [(a)] $\MPB$ satisfies the $\kappa_\eta^+$-c.c..

\item [(b)] $g^*$ is a function with domain $\eta.$

\item [(c)] For $\xi<\eta$ set $x^*_\xi=g^*(\xi)$ and $\tau_\xi=\kappa_{x^*_\xi}=x^*_\xi \cap \kappa.$ Then $\kappa_\eta=\bigcup_{\xi<\eta}x^*_\xi$
and $\kappa=\sup_{\xi<\eta}\tau_\xi.$
\end{enumerate}
\end{lemma}
We now state a factorization property of $\MPB.$ Assume $p \in \MPB$ with $g^p=\{ \langle \alpha, x \rangle \},$ where $\alpha < \eta$ is a limit ordinal.
For $\xi < \alpha$ let $v_\xi=\bar{U}^\xi_{\alpha,x}$. Also for $\xi < \zeta < \alpha$ let $y \mapsto \bar{v}^\xi_{\zeta,y}$ be such that
$v_\xi=[y \mapsto \bar{v}^\xi_{\zeta,y}]_{v_\zeta}$, where each $\bar{v}^\xi_{\zeta,y}$
is a normal measure on $P_{\kappa_x \cap y}(otp(\kappa_x \cap y))$. Also let $v^\xi_{\zeta,y}$ be the lift of
$\bar{v}^\xi_{\zeta,y}$ on $P_{\kappa_x \cap y}(\kappa_x \cap y)$.
Also we can find the sets $b_\xi \subseteq \overline{H(\xi)}, b_\xi \in v_\xi,$ such that for all $\zeta < \gamma < \xi< \alpha$ and
 all $y \in b_\xi, \bar{v}^{\zeta}_{\xi,y}=[z \mapsto \bar{v}^{\zeta}_{\gamma, z}]_{v^{\gamma}_{\xi,y}}$.
Also let $k_\xi=K^\xi_{\alpha,x}.$
\begin{lemma}
(The factorization property)
With the same notation as above, the forcing notion $\MPB / p$ can be factored as $\MPB_0 \times \MPB_1,$
where
\begin{enumerate}
\item $\MPB_0$ is defined using:
\begin{itemize}
\item The normal measures $v_\xi=\bar{U}^\xi_{\alpha,x}, \xi < \alpha,$
\item The sets $b_\xi, \xi < \alpha,$
\item The functions $y \mapsto \bar{v}^\xi_{\zeta,y},$ for $\xi < \zeta < \alpha,$
\item The generic filters $k_\xi, \xi < \alpha.$
\end{itemize}
\item $\MPB_1$ is defined using:
\begin{itemize}
\item The normal measures $U_\xi, \alpha < \xi < \eta,$
\item The sets $H(\xi), \alpha < \xi < \eta,$
\item The functions $x \mapsto \bar{U}^\xi_{\zeta,x},$ for $\alpha < \xi < \zeta < \eta,$
\item The generic filters $K_\xi, \alpha < \xi < \eta.$
\end{itemize}
\end{enumerate}
\end{lemma}

With the same notation as above, forcing with $\MPB_0$ adds a generic sequence $\langle   y_\xi: \xi<\alpha   \rangle$ such that $y_\xi \in P_{\kappa_x}(\kappa_x^{+\xi})$
and $\bigcup_{\xi<\alpha}y_\xi=\kappa_x^{+\alpha}.$ As $otp(x)=\kappa_x^{+\alpha},$ so using the resulting order isomorphism,
we can lift this chain to get a sequence $\langle   y^*_\xi: \xi<\alpha   \rangle,$ with $y^*_\xi \in P_{\kappa_x}(x \cap \kappa^{+\xi})$ whose union
is $x$. Setting $\tau^*_\xi=y_\xi \cap \kappa_x,$ we have  $\tau^*_\xi=y^*_\xi \cap \kappa=\tau_\xi$ and forcing with $\MPB_0$ collapses cardinals between $(\tau_\xi)_*^{++}$ and $\tau_{\xi+1}$ for each $\xi<\alpha.$

In general, if $p\in \MPB$ is such that $\dom(g^p)$ has size $n$, then we can factor $\MPB/p$ as the product of $n+1$ forcings as above.
\begin{lemma} (Prikry property)
\begin{enumerate}
%
\item [(a)]  $(\MPB, \leq, \leq^*)$ satisfies the Prikry property.

\item [(b)] Let $p\in \MPB, \alpha \in \dom(g^p)$ where $\alpha$ is a limit ordinal and let $\phi$ be a statement in the forcing language.
Then there is a condition $p' \leq^* p$ such that if $q \leq p$ and $q$ decides $\phi,$ then $q \upharpoonright \alpha^{\frown} p'\upharpoonright [\alpha, \eta)$
decides $\phi$ in the same way.
\end{enumerate}
\end{lemma}
Using the above lemmas, and by standard arguments, we  get the following.
\begin{lemma}
(Preservation of cardinals)
\begin{enumerate}
\item [(a)] $\kappa$ is preserved and has cofinality $cf(\eta)$ in the generic extension.

\item [(b)] All cardinals and cofinalities below $\tau_0$ are preserved.

\item [(c)] Let $\tau$ be a cardinal in $V$ with $\tau_\alpha < \tau < \tau_{\alpha}^{+\alpha+1},$ for $\alpha$ limit. Then $|\tau|^{V[H][G]}=\tau_\alpha.$

\item [(d)] Cardinals between $\kappa$ and $\kappa_\eta=\kappa^{+\eta}$ are collapsed and $\kappa^+_{V[H][G]}=\kappa^{+\eta+1}_{V[H]}=\kappa_\eta^+.$

\item [(e)] For each $\alpha < \eta,$ the cardinals between $(\tau_\alpha)_*^{++}$ and $\tau_{\alpha+1}$ are collapsed.

\item [(f)] Let $\tau<\kappa$ be a cardinal in $V$  such that for some  ordinal $\alpha < \eta, \tau_\alpha' \leq \tau \leq (\tau_\alpha)_*^{++}$, where
for successor $\alpha, \tau_\alpha'=\tau_\alpha$ and for limit $\alpha, \tau_\alpha'=\tau_{\alpha}^{+\alpha+1}.$
Then $\tau$ is preserved. Each $\tau_\alpha, \alpha < \eta$ is preserved too.
\end{enumerate}
\end{lemma}
It follows that
\[
CARD^{V[H][G]} \cap (\tau_0, \kappa)= \{\tau_\alpha : \alpha < \eta \} \cup \bigcup_{\alpha<\eta} [\tau_\alpha', (\tau_\alpha)_*^{++}],
\]
where for successor $\alpha, \tau_\alpha'=\tau_\alpha$ and for limit $\alpha, \tau_\alpha'=\tau_{\alpha}^{+\alpha+1}.$
The next lemma can be proved easily.
\begin{lemma}
The only measurable cardinals of $V[H][G]$ in the interval $(\tau_0, \kappa)$
are of the form $(\tau_\alpha)_*$ for some $\alpha<\eta.$
\end{lemma}
\begin{proof}
For $\alpha<\tau,$ we can factor the forcing as a product of two forcing notions $\MPB_0 \times \MPB_1$ where $\MPB_0$ has size less than $(\tau_\alpha)_*$
and $\MPB_1$ does not add new bounded subset to $(\tau_\alpha)_*^{++}.$ It follows that $(\tau_\alpha)_*$ remains measurable in the generic extension.
It is clear that no other measurable cardinals exist.
\end{proof}
We are now ready to complete the proof of the main theorem.
First note that:
\begin{lemma}
Let $K$ be $\Col(\omega, \tau_0^{+\eta})$-generic over $V[H][G],$ where $\tau_0$ is the first Prikry point added by $\MPB.$ Then in $V[H][G][K], \kappa$ is the limit of the first $\eta$ measurable cardinals.
\end{lemma}
\begin{proof}
In $V[H][G][K],$ the only measurable cardinals below $\kappa$ are $(\tau_\alpha)_*, \alpha<\eta,$ which are cofinal in $\kappa.$
\end{proof}
\begin{definition}
Set $\MRB= \MPB * \lusim{\Col}(\omega, \tau_0^{+\eta}).$ Given $r=\langle p, d \rangle \in \MRB,$ let $stem(r)=\langle  g^p, f^p, d  \rangle$.
Also call a triple $\langle g, f, d  \rangle$ an $\MRB$-stem if there exists $r \in \MRB$ with $stem(r)=\langle g, f, d  \rangle.$
\end{definition}
Now we show that there is some generic extension of $V[H]$ by $\MRB=\MPB * \lusim{\Col}(\omega, \tau_0^{+\eta})$ in which the tree property holds at
$\kappa^+=\kappa^{+\eta+1}=\kappa_\eta^+$, which will complete the proof of our main theorem. Note that knowing $\tau_0,$ we have
\[
\MPB * \lusim{\Col}(\omega, \tau_0^{+\eta}) \simeq  \MPB \times \Col(\omega, \tau_0^{+\eta}).
\]
Assume on the contrary that for any  $\MPB * \lusim{\Col}(\omega, \tau_0^{+\eta})$-generic filter $G*K$ over $V[H]$, the tree property fails in
$V[H][G][K]$ at $\kappa_\eta^+.$

Let $\lusim{T}$ be an $\MRB$-name which is forced by the trivial condition to be a $\kappa_\eta^+$-Aronszajn tree.
We may further suppose that the trivial condition forces ``the elements of the $\alpha$-th level of $\lusim{T}$ are elements of $\{\alpha\}\times \kappa$
for $\alpha < \kappa_\eta^{+}$''. We will show that there is a generic $G*K$ such that $T=\lusim{T}[G*K]$ has a cofinal branch in $V[H][G][K]$, a contradiction. The proof is motivated by
\cite{sinapova1} and \cite{sinapova2}.

\begin{lemma}
There is $\vec{\gamma} \in \eta^{<\omega}$ and an unbounded subset $I \subseteq \kappa_\eta^+, I \in V[H],$ such that
for all $\alpha < \beta$ in $I$, there are $\xi,\zeta<\kappa$ and a condition $r=(p, d) \in \MRB$ with $\dom(g^{p})=\vec{\gamma}$
such that $r \Vdash~ \langle \alpha, \xi \rangle \leq_{\lusim{T}} \langle \beta, \zeta  \rangle$.
\end{lemma}
\begin{proof}
Let $j: V[H]  \rightarrow M$ be a $\kappa_\eta^+$-supercompact embedding with critical point $\kappa$
and let $L^*=G^**K^*$ be $j(\MRB)$-generic over $M$, such that if $\langle x^*_\xi: \xi<\eta  \rangle$ is the Prikry sequence added by forcing, then $\tau^*_0=x^*_0 \cap j(\kappa)$ equals $\kappa$. It follows that
\[
j(\MRB)=j(\MPB) \times \Col(\omega, \kappa_\eta),
\]
and so
$\kappa_\eta^+$ is preserved in $M[L^*]$ and remains regular. Let $T^*=j(\lusim{T})[L^*].$  Let $\mu$ be such that $\sup j''\kappa_\eta^+ < \mu < j(\kappa_\eta^+),$ and let $u\in T^*_{\mu}$.
Then
\[
\forall \alpha < \kappa_\eta^+ ~\exists \xi_\alpha < j(\kappa) ~\exists r_\alpha \in L^*~(r_\alpha=(p_\alpha, d_\alpha) \Vdash~ \langle j(\alpha), \xi_\alpha \rangle <_{j(\lusim{T})} u).
\]
It follows from the regularity of $\kappa_\eta^+$ in $M[L^*]$ that there is an unbounded subset $I^* \subseteq \kappa_\eta^+, I^*\in M[L^*]$ and a fixed
$\vec{\gamma} \in \eta^{<\omega}$, such that for all $\alpha \in I^*, \dom(g^{p_\alpha})=\vec{\gamma}.$ By further shrinking of $I^*$, we can also assume that for some
$d\in \Col(\omega, \kappa^{+\eta})$ and for all $\alpha \in I^*, d_\alpha=d.$

Consider the pair $\langle g, f  \rangle$, where $\dom(g)=\dom(f)=\vec{\gamma}$, and
\begin{enumerate}
\item $g=g^{p_\alpha},$ for some and hence all $\alpha \in I^*$,
\item For  all $\xi \in \vec{\gamma}, f(\xi)=\bigcup_{\alpha \in I^*}f^{p_\alpha}(\xi).$
\end{enumerate}
It is easily seen that $\langle g, f  \rangle$ is a $\MPB$-stem. Let
\[
I=\{\alpha < \kappa_\eta^+: \exists r\in j(\MRB)~(stem(r)=\langle g,f, d \rangle \text{~and~}\exists \xi<j(\kappa) r \Vdash~\langle j(\alpha, \xi_\alpha) \rangle <_{j(\lusim{T})} u )    \}.
\]
Then $I \in V[H]$ and $I^* \subseteq I,$ so $I$ is unbounded. Then $I$ is as required.
\end{proof}
Fix $\vec{\gamma}$ and $I$ as in the conclusion of the above lemma.
Before we continue, let us introduce a notion that will be helpful.
\begin{definition}
Assume $\gamma < \eta.$

$(a)$ Let $\Xi_\gamma$ be the set of all $\Lambda = \langle \langle g_\Lambda, f_\Lambda, H_\Lambda, F_\Lambda \rangle, d_\Lambda  \rangle,$ such that for some condition $r \in \MRB,$
$stem(r)= \langle g_\Lambda, f_\Lambda, d_\Lambda \rangle$ and $r \upharpoonright \gamma = \Lambda.$ Also set
$\Xi=\bigcup_{\gamma < \eta} \Xi_\gamma.$

$(b)$ Given $\Lambda \in \Xi,$ we set $stem(\Lambda)= \langle g_\Lambda, f_\Lambda, d_\Lambda \rangle$.

$(c)$  For $\Lambda \in \Xi$ and a statement $\phi$ in the forcing language, by $\Lambda \Vdash^* \phi$ we mean there exists a
condition $r\in \MRB$ witnessing $\Lambda \in \Xi$ which forces $\phi.$ Note that this is well-defined, as any two such conditions
are compatible.

\end{definition}
\begin{lemma}
There is in $V[H],$ an unbounded set $J \subseteq \kappa_\eta^+,$ a tuple $\langle \bar{g}, \bar{f}, \bar{H}, \bar{F}  \rangle$, some fixed
$\bar{d}$
and a sequence of nodes $\langle u_\alpha: \alpha \in J \rangle$ such that $\dom(\bar{g})=\dom(\bar{f})=\vec{\gamma}$ and if we set
$\gamma_0=\max(\vec{\gamma}),$ then $\bar{H}$ and $\bar{F}$ have domain $\gamma_0 \setminus \vec{\gamma}$
and for all $\alpha < \beta$ in $J$, there is a condition $r \in \MRB$ such that
\begin{enumerate}
\item $stem(r)= \langle \bar{g}, \bar{f}, \bar{d} \rangle,$
\item $r \upharpoonright (\gamma_0+1)= \langle \langle \bar{g}, \bar{f}, \bar{H}, \bar{F}  \rangle, \bar{d} \rangle$
\item $r \Vdash$`` $u_\alpha <_{\lusim{T}} u_\beta$ ''.
\end{enumerate}
\end{lemma}
\begin{proof}
Let $j: V \rightarrow N$ be a $\kappa_\eta^+$-supercompact embedding with critical point $\kappa_{\gamma_0+3}.$ Using standard arguments we can extend $j$ to some
\[
j^*: V[H] \rightarrow N[H^*],
\]
where $j^*$ is defined in $V[H^*],$ and $H^*$ is $j(\MCB)=\MCB\times \MCB'$-generic with $\MCB'$ being $\kappa_{\gamma_0+3}$-closed.
Let $\mu \in j^*(I)$ be such that $\sup j''\kappa_\eta^+ < \mu < j(\kappa_\eta^+).$ By elementarity, for all $\alpha \in I$
we can find $\xi_\alpha, \delta_\alpha < \kappa$ and $r_\alpha =(p_\alpha, d_\alpha) \in j^*(\MRB)$
with $\dom(g^{p_\alpha})=\vec{\gamma}$
such that $r_\alpha \Vdash$`` $\langle j(\alpha), \xi_\alpha \rangle <_{j^*(\lusim{T})} \langle \mu, \delta_\alpha \rangle$''.

By simple counting arguments, and the fact that $\MCB'$ does not add any new sequences of length $\kappa_{\gamma_0},$
we can conclude that there is a cofinal subset $J \subseteq I, J \in V[H^*]$, fixed ordinals
$\xi, \delta< \kappa$ and a fixed $\langle \langle \bar{g}, \bar{f}, \bar{H}, \bar{F}  \rangle, \bar{d} \rangle$
such that for all $\alpha \in J, \xi_\alpha=\xi, \delta_\alpha=\delta$
and $r_\alpha \upharpoonright (\gamma_0+1)= \langle \langle \bar{g}, \bar{f}, \bar{H}, \bar{F}  \rangle, \bar{d} \rangle$.
Note that for any such $\alpha, stem(r_\alpha)= \langle \bar{g}, \bar{f}, \bar{d} \rangle.$

Then for all $\alpha < \beta$ in $J$, there is a condition $r \in j^*(\MRB)$ with
$stem(r)= \langle \bar{g}, \bar{f}, \bar{d} \rangle$ and $r \upharpoonright (\gamma_0+1)= \langle \langle \bar{g}, \bar{f}, \bar{H}, \bar{F}  \rangle, \bar{d} \rangle$
such that
 $r \Vdash$`` $\langle j(\alpha), \xi \rangle <_{j^*(\lusim{T})} \langle j(\beta), \xi \rangle$''. Since
 $j^*( \langle \langle \bar{g}, \bar{f}, \bar{H}, \bar{F}  \rangle, \bar{d} \rangle)= \langle \langle \bar{g}, \bar{f}, \bar{H}, \bar{F}  \rangle, \bar{d} \rangle,$ by elementarity,
 there is a condition $r \in \MRB$
 such that  $stem(r)= \langle \bar{g}, \bar{f}, \bar{d} \rangle,$  $r \upharpoonright (\gamma_0+1)= \langle \langle \bar{g}, \bar{f}, \bar{H}, \bar{F}  \rangle, \bar{d} \rangle$
and $r \Vdash$`` $\langle \alpha, \xi \rangle <_{\lusim{T}} \langle \beta, \xi \rangle$''.

Working in $V[H],$ we consider the narrow system $\mathcal{S}= \langle I, \mathcal{R}  \rangle$ of height $\kappa_\eta^+$
and levels of size $\kappa,$ where
\begin{itemize}
\item $\mathcal{R}=\{ R_{\Lambda}: \Lambda \in \Xi_{\gamma_0+1}, \max(\dom(g_\Lambda))=\gamma_0         \}$,

\item For nodes $a, b,$ we have
\[
\langle a, b  \rangle \in R_{\Lambda} \Leftrightarrow \Lambda \Vdash^* a <_{\lusim{T}} b,
\]

\end{itemize}
Note that $|\mathcal{R}| < \kappa_{\gamma_0+2}$. For any $R_{\Lambda} \in \mathcal{R},$ consider the branch
\[
b_{R_{\Lambda}, \delta} = \{ \langle \alpha, \xi \rangle: \Lambda \Vdash^* \langle j(\alpha), \xi    \rangle <_{j^*(\lusim{T})}   \langle  \mu, \delta  \rangle     \}.
\]
Applying the preservation Lemma 2.3 to $\mathcal{S}$ for $\MCB'$ which is $\kappa_{\gamma_0+3}$-closed in $V[H]$,
we get that $\mathcal{S}$ has an unbounded branch in $V[H]$. I.e., in $V[H]$, there is an unbounded
$J \subseteq I, \alpha \mapsto \xi_\alpha$ and $\Lambda = \langle \langle \bar{g}, \bar{f}, \bar{H}, \bar{F}  \rangle, \bar{d} \rangle \in \Xi_{\gamma_0+1}$ with
 $\max(\dom(g_\Lambda))=\gamma_0$ such that for all
$\alpha < \beta$ in $J$ we have $\Lambda \Vdash^*$`` $\langle \alpha, \xi_\alpha \rangle <_{\lusim{T}}      \langle \beta, \xi_\beta  \rangle$''.
Setting $u_\alpha= \langle \alpha, \xi_\alpha  \rangle$ for $\alpha \in J$ we get that for
$\alpha < \beta$ in $J$, there is a condition $r \in \MRB$ which satisfies the required properties of the lemma.
\end{proof}
Fix $J \subseteq \kappa_\eta^+,$  $\langle \langle \bar{g}, \bar{f}, \bar{H}, \bar{F}  \rangle, \bar{d} \rangle$
and the sequence of nodes $\langle u_\alpha: \alpha \in J \rangle$ as in the conclusion of the above lemma.
By shrinking $J$, if necessary, we can assume that for some $\epsilon < \kappa,$ each $u_\alpha= \langle \alpha, \epsilon  \rangle$.
\begin{definition}
Assume $\Lambda, \Delta \in \Xi_\gamma$ be such that $stem(\Lambda)=stem(\Delta).$
Then $\Lambda \wedge \Delta \in \Xi_\gamma$ is defined by:

\begin{itemize}
\item $stem(\Lambda \wedge \Delta)=stem(\Lambda)=stem(\Delta),$

\item   $H_{\Lambda \wedge \Delta} = H_{\Lambda} \wedge H_{\Delta},$ i.e., for all $\sigma \in \dom(H_{\Lambda \wedge \Delta}),$
\begin{center}
$ H_{\Lambda \wedge \Delta}(\sigma) =
\{y \in  H_{\Lambda}(\sigma) \cap H_{\Delta}(\sigma): F_\Lambda(\sigma)(y)$ is compatible with $F_\Delta(\sigma)(y) \}$.
\end{center}
\item   $F_{\Lambda \wedge \Delta} = F_{\Lambda} \vee F_{\Delta},$ i.e., for all $\sigma \in \dom(F_{\Lambda \wedge \Delta})$ and all $y \in \dom(F_{\Lambda \wedge \Delta}(\sigma))=H_{\Lambda \wedge \Delta}(\sigma)$,
\[
F_{\Lambda \wedge \Delta}(\sigma)(y)= F_{\Lambda}(\sigma)(y) \cup F_{\Delta}(\sigma)(y).
\]
\end{itemize}
\end{definition}

\begin{lemma}
Suppose that $\Lambda_0=\langle  \langle g, f, H, F \rangle, d \rangle \in \Xi_{\gamma_0}$, $L \subseteq \kappa_\eta^+$ is unbounded and for all
$\alpha < \beta$ in $L$, we have $\Lambda_0 \Vdash^*$``$u_\alpha <_{\lusim{T}} u_\beta$''. Let $\gamma_0 \leq \gamma < \eta.$ Then there are $\rho< \kappa_\eta^+$
and $\langle \Lambda_\alpha: \alpha \in L \setminus \rho       \rangle$ such that:
\begin{enumerate}
\item $\Lambda_\alpha \in \Xi_{\gamma+1}, stem(\Lambda_\alpha)=stem(\Lambda_0)$ and $\Lambda_\alpha \upharpoonright \gamma_0=\Lambda_0,$


\item For all $\alpha < \beta$ in $L \setminus \rho$ and for all $y \in H_{\Lambda_\alpha}(\gamma) \cap H_{\Lambda_\beta}(\gamma)$ such that  $F_{\Lambda_\alpha}(\gamma)(y)$ and $F_{\Lambda_\beta}(\gamma)(y)$
are compatible, we have
\[
 \langle \langle  g ^{\frown} \langle \gamma, y \rangle , f ^{\frown} \langle \gamma, F_{\Lambda_\alpha}(\gamma)(y) \cup F_{\Lambda_\beta}(\gamma)(y) \rangle, H_{\Lambda_\alpha} \wedge H_{\Lambda_\beta} \upharpoonright \gamma, F_{\Lambda_\alpha} \vee F_{\Lambda_\beta} \upharpoonright \gamma \rangle,  d                          \rangle \Vdash^* u_\alpha <_{\lusim{T}} u_\beta.
\]
\end{enumerate}
\end{lemma}
\begin{proof}
Let $j: V \rightarrow N$ be a $\kappa_\eta^+$-supercompact embedding with critical point $\kappa_{\gamma+4}.$ Using standard arguments we can extend $j$ to some
\[
j^*: V[H] \rightarrow N[H^*],
\]
where $j^*$ is defined in $V[H^*],$ and $H^*$ is $j(\MCB)=\MCB\times \MCB'$-generic with $\MCB'$ being $\kappa_{\gamma+4}$-closed.
Let $\mu \in j^*(L)$ be such that $\sup j''\kappa_\eta^+ < \mu < j(\kappa_\eta^+).$
Note that $j^*(\Lambda_0)= \Lambda_0$. By elementarity, for each $\alpha \in L$, we can find $r_\alpha = (p_\alpha, d) \in j^*(\MRB)$
with $stem(r_\alpha)= \langle g, f, d \rangle$ and $r_\alpha \upharpoonright \gamma_0 = \Lambda_0$ such that
\[
r_\alpha \Vdash~\langle   j(\alpha), \epsilon \rangle <_{j^*(\lusim{T})} \langle \mu, \epsilon  \rangle.
\]
For each $\alpha \in L$, denote $r_\alpha = \langle  \langle g, f, H_\alpha, F_\alpha     \rangle      ,d  \rangle$. Then $H_\alpha(\gamma) \in j^*(U_\gamma)=U_\gamma$
and $j^*(H_\alpha(\gamma))=H_\alpha(\gamma)$. By counting arguments, we can find an unbounded subset of $L$ on which
$\alpha \mapsto r_\alpha \upharpoonright \gamma+1 $ is constant. So we can find an unbounded subset $\bar{L} \subseteq L, \bar{L} \in V[H^*],$
and $\bar{\Lambda}$ such that for all $\alpha \in \bar{L}, r_\alpha \upharpoonright \gamma+1=\bar{\Lambda}.$ Note that $\bar{\Lambda} \in \Xi_{\gamma+1}$
and $\bar{\Lambda} \upharpoonright \gamma_0=\Lambda_0.$

Working in $V[H],$ consider the narrow system $\mathcal{S}= \langle L, \mathcal{R}  \rangle$ of height $\kappa_\eta^+$
and levels of size $\kappa,$ where
\begin{itemize}
\item $\mathcal{R}=\{ R_{\Lambda}: \Lambda \in \Xi_{\gamma+1}, \Lambda \upharpoonright \gamma_0 =\Lambda_0$ and  $stem(\Lambda)=stem(\Lambda_0)   \}$,

\item For nodes $a, b,$ we have
\[
\langle a, b  \rangle \in R_{\Lambda} \Leftrightarrow \Lambda \Vdash^* a <_{\lusim{T}} b,
\]

\end{itemize}
Note that $|\mathcal{R}| < \kappa_{\gamma+3}$. For any $R_{\Lambda} \in \mathcal{R},$ consider the branch
\[
b_{R_{\Lambda}} = \{ \langle \alpha, \xi \rangle: \Lambda \Vdash^* \langle j(\alpha), \epsilon    \rangle <_{j^*(\lusim{T})}   \langle  \mu, \epsilon  \rangle     \}.
\]
Applying the preservation Lemma 2.3 to $\mathcal{S}$ for $\MCB'$ which is $\kappa_{\gamma+4}$-closed in $V[H]$,
we get that $\mathcal{S}$ has an unbounded branch in $V[H]$. I.e., in $V[H]$, there is an unbounded
$J \subseteq L$ and $\Lambda^* \in \Xi_{\gamma+1} $ with $\Lambda^* \upharpoonright \gamma_0 =\Lambda_0$ and  $stem(\Lambda^*)=stem(\Lambda_0)$,
 such that for all
$\alpha < \beta$ in $J$ we have $\Lambda^* \Vdash^*$`` $\langle \alpha, \epsilon \rangle <_{\lusim{T}}      \langle \beta, \epsilon  \rangle$''.

Let $\rho=\min(J).$ For every $\alpha \in (L \setminus \rho) \setminus J$ let $\alpha'=\min(J\setminus \alpha)$
and let $\Lambda^*_\alpha \in \Xi_{\gamma+1}$ be such that $\Lambda^*_\alpha \upharpoonright \gamma_0=\Lambda_0, stem(\Lambda^*_\alpha)=stem(\Lambda_0)$ and $\Lambda^*_\alpha \Vdash^*$`` $\langle \alpha, \epsilon \rangle <_{\lusim{T}}      \langle \alpha', \epsilon  \rangle$''. We can find such
a $\Lambda^*_\alpha$ by our assumption
$\Lambda_0 \Vdash^*$``$u_\alpha <_{\lusim{T}} u_{\alpha'}$''.

Let $\alpha \in L \setminus \rho.$ If $\alpha \in J$ define $\Lambda_\alpha =\Lambda^*_\alpha.$
If $\alpha \notin J,$ then set $\Lambda_\alpha=\Lambda^*_{\alpha'}  \wedge \Lambda^*.$
Then $\rho$ and $\langle   \Lambda_\alpha: \alpha \in L \setminus \rho        \rangle$
are as required.
\end{proof}
\begin{lemma}
There are $\rho< \kappa_\eta^+$ and conditions $\langle p_\alpha: \alpha \in J \setminus \rho     \rangle$
with $p_\alpha \upharpoonright \gamma_0+1 = \langle \langle \bar{g}, \bar{f}, \bar{H}, \bar{F}  \rangle, \bar{d} \rangle$
such that for all $\alpha < \beta$ in $J \setminus \rho, p_\alpha \wedge p_\beta \Vdash$``$u_\alpha <_{\lusim{T}} u_\beta$''.
Here $ p_\alpha \wedge p_\beta $ denotes the weakest extension of $p_\alpha$ and $p_\beta.$
\end{lemma}
\begin{proof}
Let $\Pi= \langle \langle \bar{g}, \bar{f}, \bar{H}, \bar{F}  \rangle, \bar{d} \rangle.$
We define sequences $\langle \rho_\xi: \gamma_0 < \xi< \eta   \rangle$ and $\langle \Lambda^\xi_\alpha: \alpha \in J \setminus \rho_\xi,   \gamma_0 < \xi< \eta     \rangle$ by induction on $\xi$ such that for all such $\xi:$
\begin{enumerate}
\item [$(1)_\xi$:] For all $\alpha \in J \setminus \rho_\xi, \Lambda^\xi_\alpha \in \Xi_{\xi}$,

\item [$(2)_\xi$:] For all $\zeta < \xi$ and all $\alpha \in J \setminus \rho_\xi$ we have $\Lambda^\xi_\alpha \upharpoonright \zeta = \Lambda^\zeta_\alpha,$
\item [$(3)_\xi$:] For all $\alpha < \beta$ in $J \setminus \rho_{\xi+1}$, for all $\Lambda^* \in \Xi_{\xi+1}$ with $stem(\Lambda^*)$
extending $stem(\Pi)$ with $\max(\dom(g_{\Lambda^*}))=\xi,$ if $\Lambda^*$ is compatible with $\Lambda^{\xi+1}_\alpha$
and $\Lambda^{\xi+1}_\beta,$
then $\Lambda^* \Vdash^*$``$u_\alpha <_{\lusim{T}} u_\beta$''.
\end{enumerate}

For $\xi=\gamma_0+1$ set $\rho_{\gamma_0+1}=0$ and $\Lambda^{\gamma_0+1}_\alpha=\Pi,$ for all $\alpha \in J.$

For limit $\xi > \gamma_0$ set $\rho_\xi=\sup_{\gamma_0 < \zeta < \xi} \rho_\zeta,$
and for $\alpha \in J \setminus \rho_\xi$ let $\Lambda^\xi_\alpha$ be the unique element of $\Xi_\xi$
such that for all $\gamma_0 < \zeta < \xi, \Lambda^\xi_\alpha \upharpoonright \zeta = \Lambda^\zeta_\alpha.$
Note that this is well-defined by assumptions $(2)_\zeta, \gamma_0 < \zeta < \xi.$
Clearly $(1)_\xi$ and $(2)_\xi$ are satisfied and there is noting to prove related to $(3)_\xi$.

Now suppose that we have defined $\langle \rho_\zeta: \gamma_0 < \zeta \leq \xi \rangle $ and $ \langle \Lambda^\zeta_\alpha: \alpha \in J \setminus \rho_\zeta, \gamma_0 < \zeta \leq \xi  \rangle.$ We define $\rho_{\xi+1}$ and
$ \langle \Lambda^{\xi+1}_\alpha: \alpha \in J \setminus \rho_{\xi+1} \rangle.$

Call $\Lambda \in \Xi_\zeta, \gamma_0 <\zeta \leq \xi$ fits $\alpha$, if $stem(\Lambda)$
extends $stem(\Pi)$ and $\Lambda $ is compatible with $\Lambda^{\varsigma}_\alpha, \gamma_0< \varsigma \leq \zeta$.
For any $\Lambda \in \Xi_\zeta, \zeta \leq \xi$ with $stem(\Lambda)$
extending  $stem(\Pi)$ set
\[
J^{\Lambda} = \{ \alpha \in J \setminus \rho_\xi: \Lambda \text{~fits~} \alpha          \}.
\]
We define a function $\Lambda \mapsto \rho^{\Lambda}$ on all $\Lambda$'s
as above as follows:
\begin{itemize}
\item If $J^{\Lambda}$ is bounded in $\kappa_\eta^+$, then $\rho^{\Lambda} < \kappa_\eta^+$ is a bound,

\item Otherwise, let $\rho^{\Lambda}$ and $\langle \Pi^{\Lambda}_\alpha: \alpha \in J^\Lambda \setminus \rho^{\Lambda}   \rangle$ be given by Lemma 3.19, applied to $\Lambda$ and $J^{\Lambda}.$
\end{itemize}
Let $\rho_{\xi+1}=\sup\{ \rho^\Lambda:  \Lambda \in \Xi_\zeta, \zeta \leq \xi, stem(\Lambda)$
extends $stem(\Pi)     \}$. Note that $\rho_{\xi+1} < \kappa_\eta^+.$
For $\alpha \in J \setminus \rho_{\xi+1}$ define $\Lambda^{\xi+1}_\alpha$ so that
\begin{enumerate}
\item $\Lambda^{\xi+1}_\alpha \upharpoonright \xi = \Lambda^{\xi}_\alpha$,

\item $\xi \in \dom(H_{\Lambda^{\xi+1}_\alpha})=\dom(F_{\Lambda^{\xi+1}_\alpha})$,

\item If $X= \{ \Lambda:  \Lambda \in \Xi_\zeta, \zeta \leq \xi$ and $\Lambda$ fits $\alpha   \}$, then
\[
H_{\Lambda^{\xi+1}_\alpha}(\xi)= \bigtriangleup_{\Lambda \in X} H_{\Pi^{\Lambda}_\alpha}(\xi),
\]
and
\[
 [F_{\Lambda^{\xi+1}_\alpha}(\xi)]_{U_\xi} = \bigcup_{\Lambda \in X}[F_{\Pi^{\Lambda}_\alpha}(\xi)]_{U_\xi}.
\]

\item By shrinking $H_{\Lambda^{\xi+1}_\alpha}(\xi)$, we can assume that for all $y \in H_{\Lambda^{\xi+1}_\alpha}(\xi),$
we have
\[
F_{\Lambda^{\xi+1}_\alpha}(\xi)(y) = \bigcup_{\Lambda \prec y}F_{\Pi^{\Lambda}_\alpha}(\xi)(y).
\]
\end{enumerate}
Now we have to verify the above conditions are satisfied. Conditions $(1)_\xi$ and $(2)_\xi$
are easily seen to be true. For $(3)_\xi$, assume that $\alpha < \beta$ are in $J \setminus \rho_\xi$ and
$\Lambda^* \in \Xi_{\xi+1}$
is such that $stem(\Lambda^*)$ extends $stem(\Pi)$ with $\max(\dom(g_{\Lambda^*}))=\xi.$ Also assume that $\Lambda^*$ is compatible with $\Lambda^{\xi+1}_\alpha$
and $\Lambda^{\xi+1}_\beta.$ Consider $\Lambda = \Lambda^* \upharpoonright \xi.$
Then $\Lambda$ fits both $\alpha$ and $\beta,$ so by our construction and Lemma 3.19, we can easily conclude that
\[
 \langle \langle  g_\Lambda ^{\frown} \langle \xi, g_{\Lambda^*}(\xi) \rangle , f_\Lambda ^{\frown} \langle \xi, F_{\Lambda^{\xi+1}_\alpha}(\xi)(y) \cup F_{\Lambda^{\xi+1}_\beta}(\xi)(y) \rangle, H_{\Lambda^{\xi+1}_\alpha} \wedge H_{\Lambda^{\xi+1}_\beta}\upharpoonright \xi, F_{\Lambda^{\xi+1}_\alpha} \vee F_{\Lambda^{\xi+1}_\beta} \upharpoonright \xi \rangle, d                          \rangle \Vdash^* u_\alpha <_{\lusim{T}} u_\beta.
\]
This gives
$\Lambda^* \Vdash^*$``$u_\alpha <_{\lusim{T}} u_\beta$'', and the result follows.

Finally set $\rho=\sup_{\gamma_0 < \xi < \eta} \rho_\xi < \kappa_\eta^+$
 and for each $\alpha \in J \setminus \rho$ let $p_\alpha$ be the unique condition
 obtained using the sequence $\langle  \Lambda^\xi_\alpha: \gamma_0 < \xi < \eta           \rangle$. So
 \begin{itemize}
 \item $stem(p_\alpha)=stem(\Pi),$
 \item $p_\alpha \upharpoonright \gamma_0+1= \Pi,$
 \item For all $\xi$ with $\gamma_0 < \xi < \eta, \xi \in \dom(H^{p_\alpha})$,
 \[
 H^{p_\alpha}(\xi) = H_{\Lambda^{\xi+1}_\alpha}(\xi),
 \]
 and
 \[
 F^{p_\alpha}(\xi) = F_{\Lambda^{\xi+1}_\alpha}(\xi).
 \]
 \end{itemize}

If $q \leq p_\alpha \wedge p_\beta$,
 then by construction $stem(q) \Vdash^*$``$u_\alpha <_{\lusim{T}} u_\beta$'', so
$q \nVdash$``$u_\alpha \nless_{\lusim{T}} u_\beta$''. It follows that $p_\alpha \wedge p_\beta \Vdash$``$u_\alpha <_{\lusim{T}} u_\beta$''.
The lemma follows.
\end{proof}

The rest of the argument is standard.  Let $\langle p_\alpha: \alpha \in J \setminus \rho     \rangle$ be as in the above Lemma.
We want to find a condition $q$ that forces that the set $\{\alpha \in J: p_\alpha\in G\}$ is unbounded in $\kappa_\eta^+$, since such a $q$ forces that there is a branch in $\lusim{T}$, namely, the downward closure of $\{\langle \alpha , \epsilon \rangle : p_\alpha \in G\}$.

Assume that this is not the case. So, it is forced that the set $\{\alpha \in J : p_\alpha \in G\}$ is bounded. Since $\MRB$ is $\kappa_\eta^+$-c.c., in $V[H]$, there is an ordinal $\beta < \kappa_\eta^+$ such that it is forced by the trivial condition that $\{\alpha : p_\alpha \in G\}\subseteq \beta$. This is impossible, as for any $\gamma \in J \setminus \beta$, $p_\gamma$ forces the opposite statement.

This completes the proof of the theorem in the case $\eta < \kappa_0$.
We now consider the general case. Fix the sequence $\langle \kappa_\xi: \xi \leq \eta     \rangle$
as before. As $\eta < \kappa_\eta,$ we can find the least $\xi_* < \eta$
such that $\kappa_{\xi_*-1} \leq\eta < \kappa_{\xi_*}$, where $\kappa_{-1}=\omega$. Note that we can assume that $\xi_* >0.$
 We consider two cases:
 \begin{enumerate}
 \item $\eta > \kappa_{\xi_*-1}:$ Then by the same arguments as above, we can find a generic extension in which cardinals below $\kappa_{\xi_*-1}$
 are preserved, and $\kappa_\eta^+$ is the successor of the limit of the first $\eta - \kappa_{\xi_*-1}$
 measurable cardinals above $\kappa_{\xi_*-1}$ which has the tree property.
 But then $\kappa_\eta^+$ is the successor of the limit of the first $\eta$
 measurable cardinals and it has the tree property.

 \item $\eta = \kappa_{\xi_*-1}:$ Then $\eta - \xi_* =\eta,$ so again by the same arguments as above we can find a generic extension in which
 $\kappa_{\xi_*}$ changes its cofinality to $\eta, \kappa_\eta^+$ becomes the new successor of $\kappa_{\xi_*}$,
 $\kappa_{\xi_*}$ is the limit of the first $\eta$ measurable cardinals and $\kappa_\eta^+$ has the tree property.
 \end{enumerate}
 The theorem follows.
\hfill$\Box$
\begin{remark}
By the same argument, we can prove the following: Assume $\lambda$ is a singular limit of $\eta$ supercompact cardinals, where $\eta \leq \lambda$
is a limit ordinal. Then there is a generic extension in which $\lambda^+=\aleph_{\eta^2+1},$ and the tree property holds at $\aleph_{\eta^2+1}.$
\end{remark}

\section{Tree property and the failure of $SCH$}

By a result of Neeman \cite{neeman0}, it is consistent, relative to the existence of infinitely many supercompact  cardinals, that for some singular cardinal $\kappa$
 of countable cofinality, $SCH$ fails at $\kappa$ and the tree property holds at $\kappa^+$. Sinapova \cite{sinapova3} proved the same result for
 $\kappa=\aleph_{\omega^2}.$ She also proved similar result for a singular cardinal of uncountable cofinality \cite{sinapova2}.
 In this section we prove Theorem 1.2, which extends the above results of Neeman and Sinapova.
The basic idea of the proof is to combine the ideas of Section 3 with those of Sinapova \cite{sinapova3}. As the proof is similar to the above proof of Theorem 1.1,
we just mention the main changes which are needed to get the required result.

Fix an increasing continuous sequence $\langle \kappa_\xi: \xi \leq \eta  \rangle$ of supercompact cardinals and their limits as before, with $\eta < \kappa_0 = \kappa$
and $\lambda= \kappa_\eta.$
Let $\mathbb{C}$ be the Easton support iteration of $\Col(\kappa_i^{+}, < \kappa_{i+1}), i<\eta,$ and let $H$ be $\mathbb{C}$-generic over $V$.
Now force with $\Add(\kappa, \kappa_\eta^{++})_{V[H]}$
and let $E$ be $\Add(\kappa, \kappa_\eta^{++})_{V[H]}$-generic over $V[H]$. The next lemma can be proved as in Lemma 3.4. 
\begin{lemma}
In $V[H][E],$ there are sequences $\langle U_\xi: \xi<\eta     \rangle$ 
 and $\langle K_\xi: \xi<\eta  \rangle$ such that:
\begin{enumerate}
\item [(a)] $U_\xi$ is a normal measure on $P_\kappa(\kappa^{+\xi})$,

\item [(b)] $\kappa$ is $\kappa^{+\xi}$-supercompact in $N_\xi=Ult(V, U_\xi),$

\item [(c)] The sequence $\langle U_\xi: \xi<\eta     \rangle$  is Mitchell increasing, i.e., $\zeta<\xi<\eta \Rightarrow U_\zeta \in N_\xi,$


\item [(d)] $K_\xi$ is $\Col(\kappa^{+\eta+2}, < j_{U_\xi}(\kappa))_{N_\xi}$-generic over $N_\xi$.
\end{enumerate}
\end{lemma}
\begin{proof}
We sketch the proof of the lemma and refer to \cite{sinapova} Section 3.1, for details. The lemma is proved in several steps. For $\gamma < \kappa_\eta^{++}$
let $g_\gamma: \kappa \to \kappa$ denote the $\gamma$-th generic function added by $E$.
\\
{\bf Step 1.} Assume $\xi < \eta$ and $\mathcal{X} \subseteq P(P_\kappa(\kappa^{+\xi})), \mathcal{X} \in V[H][E].$ Then there exists a normal measure $U$ on $P_\kappa(\kappa_\eta^+)$ such that $\mathcal{X} \in \Ult(V[H][E], U)$.

To see this, assume on the contrary that there is no such a normal measure $U$. Let $\phi(\mathcal{X}, \kappa, \xi, \eta)$ be the assertion ``
$\mathcal{X} \subseteq P(P_\kappa(\kappa^{+\xi}))$ and for all normal measures $U$ on  $P_\kappa(\kappa_\eta^+),$  $\mathcal{X} \notin \Ult(V[H][E], U)$''.
so for some $\mathcal{X} \in V[H][E], \phi(\mathcal{X}, \kappa, \xi, \eta)$ holds in $V[H][E].$

Let $j: V[H][E] \to M^*$ be a $\kappa_\eta^{++}$-supercompactness embedding with critical point $\kappa$. Then $M^* \models$``$\exists \mathcal{X}, \phi(\mathcal{X}, \kappa, \xi, \eta)$''. Let $U$ be the normal measure on $P_\kappa(\kappa_\eta^+)$  generated by $j$ and let $k: \Ult(V[H][E], U) \to M^*$
be the natural embedding with $k \circ j_U=j$. By elementarity of $k$ and the fact that $k(\kappa)=\kappa,$
we have $\Ult(V[H][E], U)\models$``$\exists \mathcal{X}, \phi(\mathcal{X}, \kappa, \xi, \eta)$''. Let $\mathcal{X}'$ witness this. We have $k(\mathcal{X}')=\mathcal{X}'$ and so by elementarity of $k$, $M^* \models$``$\phi(\mathcal{X}', \kappa, \xi, \eta)$''. But $\mathcal{X}' \in \Ult(V[H][E], U)$
and $U \in M^*,$ a contradiction.
\\
{\bf Step 2.} Assume $\xi < \eta$ and $\mathcal{X} \subseteq P(P_\kappa(\kappa^{+\xi})), \mathcal{X} \in V[H][E].$ Then there exists a $\kappa_\eta^+$-supercompactness embedding $j^*: V[H][E] \to M[H^*][E^*]$  with $crit(j^*)=\kappa$ such that $\mathcal{X} \in M[H^*][E^*]$
and $j^*(\kappa) = \{j^*(g_\gamma)(\kappa): \gamma < \kappa_\eta^{++}         \}$.

The proof uses ideas of \cite{gitik-sharon}. Let $\pi: V[H][E] \to M[H^*][E']$ be a a $\kappa_\eta^+$-supercompactness embedding  with critical point $\kappa$ such that $\mathcal{X} \in M[H^*][E']$ We also assume that $\pi$ extends some $\kappa_\eta^+$-supercompactness embedding $j: V[H] \to M[H^*].$

Working in $V[H][E],$ enumerate $j(\kappa)$ as $\langle y_\gamma: \gamma < \kappa_\eta^{++}  \rangle$. also let $E'_\gamma$ be the restriction
of $E'$ to $j(\gamma)$. By ideas from Woodin's surgery method we can modify $E'_\gamma$
to find some $E^*_\gamma$ which is still generic for the related forcing notion, $E^*=\bigcup_{\gamma < \kappa_\eta^{++}} E^*_\gamma$
is $j(\Add(\kappa, \kappa_\eta^{++}))$-generic over $M[H^*]$ so that we can lift $j$ to some $j^*: V[H][E] \to M[H^*][E^*]$
and such that for $\gamma < \kappa_\eta^{++}, j^*(g_\gamma)(\kappa)=y_\gamma.$ Also it is easily seen that $\mathcal{X} \in M[H^*][E^*]$, which completes the proof.
\\
{\bf Step 3.} Assume $\xi < \eta$ and $\mathcal{X} \subseteq P(P_\kappa(\kappa^{+\xi})), \mathcal{X} \in V[H][E].$ Then there is a normal measure
$U_\xi$ on $P_\kappa(\kappa^{+\xi})$ such that $\mathcal{X} \in N_\xi=\Ult(V[H][E], U_\xi)$. Further, there is $K_\xi \in V[H][E]$
which is  $\Col(\kappa^{+\eta+2}, < j_{U_\xi}(\kappa))_{N_\xi}$-generic over $N_\xi$.

This follows easily from step 2. Let $j^*: V[H][E] \to M[H^*][E^*]$ be a  $\kappa_\eta^+$-supercompactness embedding  with critical point $\kappa$ as in Step 2 
and let $U_\xi$ be the normal measure on $P_\kappa(\kappa^{+\xi})$ derived from $j^*$, i.e.,
\[
U_\xi= \{ X \subseteq P_\kappa(\kappa^{+\xi}): j[\kappa^{+\xi}] \in j^*(X)  \}.
\]
Then $U_\xi$ is as required. The construction of $K_\xi$ is also by standard arguments.

Now by induction, we can easily construct the sequences 
$\langle U_\xi: \xi<\eta     \rangle$ 
 and $\langle K_\xi: \xi<\eta  \rangle$ with the required properties. This completes the proof of the lemma.
 \end{proof}
Now define the sets $X_\xi, \xi < \eta$, the measures $U^{\zeta}_{\xi, x}$ and the generic filters $K^{\zeta}_{\xi, x}$, for $\zeta < \xi < \eta$ and $x \in U_\xi$
as before. The forcing notion $\MPB$
is defined over $V[H][E]$ as in Definition 3.5, with the following changes:
\begin{itemize}
\item In $(7$-$1)$, require $f(\xi) \in \Col(\kappa_{g(\xi)}^{+\eta+2}, < \kappa_{g(\zeta)}),$
\item In $(7$-$2)$, require $f(\xi) \in \Col(\kappa_{g(\xi)}^{+\eta+2}, < \kappa),$
\item In $(9$-$1)$, require $F(\xi)(y) \in \Col(\kappa_{y}^{+\eta+2}, < \kappa_x),$
\item In $(9$-$2)$, require $F(\xi)(y) \in \Col(\kappa_{y}^{+\eta+2}, < \kappa).$
\end{itemize}
The order relation is defined as before (see Definition 3.7).
Finally set $\mathbb{R}=\MPB* \lusim{Col}(\omega, \tau_0^{+\eta})$. Let
$G*K $ be $\mathbb{R}$-generic over $V[H][E]$ and let $\langle  \tau_\xi: \xi < \eta    \rangle$
be the Prikry sequence added by $G$.
An analysis as in previous section shows that $\kappa$ is preserved in the resulting generic extension,
 $\kappa = \aleph_{\eta^2}^{V[H][E][G][k]}$ and $\lambda^+= \kappa_\eta^+= \aleph_{\eta^2+1}^{V[H][E][G][K]}$. Also
 $V[H][E][G][k] \models$``$2^{\aleph_{\eta^2}} \geq \kappa_\eta^{++} =\aleph_{\eta^2+2}$
 and so $SCH$ fails at $\aleph_{\eta^2}$''.

It remains to show that there is an extension of $V[H][E]$ by $\mathbb{R}$ in which tree property holds at $\kappa_\eta^+.$ Assume on the contrary that
for any $\mathbb{R}$-generic filter $G*K$ over $V[H][E],$ the tree property fails in $V[H][E][G][K]$ at $\kappa_\eta^+.$
Let $\lusim{T} \in V[H][E]$ be an $\mathbb{R}$-name which is forced by the trivial condition to be a $\kappa_\eta^+$-Aronszajn tree.
We may further suppose that the trivial condition forces ``the elements of the $\alpha$-th level of $\lusim{T}$ are elements of $\{\alpha\}\times \kappa$
for $\alpha < \kappa_\eta^{+}$''. We will show that $T=\lusim{T}[G*K]$ has a cofinal branch in $V[H][E][G][K]$.
The next lemma is similar to Lemma 3.15 and its proof is the same.

\begin{lemma}
There is $\vec{\gamma} \in \eta^{<\omega}$ and an unbounded subset $I \subseteq \kappa_\eta^+, I \in V[H][E],$ such that
for all $\alpha < \beta$ in $I$, there are $\xi,\zeta<\kappa$ and a condition $r=(p, d) \in \MRB$ with $\dom(g^{p})=\vec{\gamma}$
such that $r \Vdash~ \langle \alpha, \xi \rangle \leq_{\lusim{T}} \langle \beta, \zeta  \rangle$.
\end{lemma}
Fix $\vec{\gamma}$ and $I$ as above lemma.
\begin{lemma}
There is in $V[H][E],$ an unbounded set $J \subseteq \kappa_\eta^+,$ a tuple $\langle \bar{g}, \bar{f}, \bar{H}, \bar{F}  \rangle$, some fixed
$\bar{d}$
and a sequence of nodes $\langle u_\alpha: \alpha \in J \rangle$ such that $\dom(\bar{g})=\dom(\bar{f})=\vec{\gamma}$ and if we set
$\gamma_0=\max(\vec{\gamma}),$ then $\bar{H}$ and $\bar{F}$ have domain $\gamma_0 \setminus \vec{\gamma}$
and for all $\alpha < \beta$ in $J$, there is a condition $r \in \MRB$ such that
\begin{enumerate}
\item $stem(r)= \langle \bar{g}, \bar{f}, \bar{d} \rangle,$
\item $r \upharpoonright (\gamma_0+1)= \langle \langle \bar{g}, \bar{f}, \bar{H}, \bar{F}  \rangle, \bar{d} \rangle$
\item $r \Vdash$`` $u_\alpha <_{\lusim{T}} u_\beta$ ''.
\end{enumerate}
\end{lemma}
\begin{proof}
The proof is similar to the proof of Lemma 3.17, so we just sketch it. Let $j: V \rightarrow N$ be a $\kappa_\eta^+$-supercompact embedding with critical point $\kappa_{\gamma_0+3}.$ Using standard arguments we can extend $j$ to some
\[
j_1: V[H] \rightarrow N[H^*],
\]
where $H^*$ is $j(\MCB)$-generic over $V$. We can arrange so that $H^*=H*E*H'$, where $H'$
is $\MCB'$-generic over $V[H][E],$ for some $\MCB'$ which is
 $\kappa_{\gamma_0+3}$-distributive in $V[H][E]$ and it is
$\kappa_{\gamma_0+3}$-closed in $V[H].$ Let $F$ be $\Add(\kappa, j(\kappa_\eta^{++}))$-generic over $V[H][E].$
Then $F$ is also generic over $V[H][E][H']$. Define
\[
E^*=\{f \in \Add(\kappa, j(\kappa_\eta^{++})): f \upharpoonright j''\kappa_\eta^{++} \in j_1''E, f \upharpoonright (j(\kappa_\eta^{++}) \setminus j''\kappa_\eta^{++}) \in F           \}.
\]
Then $E^*$ is $\Add(\kappa, j(\kappa_\eta^{++}))$-generic over $V[H^*]=V[H][E][H']$ such that
$j_1''E \subseteq E^*$, and so we can extend $j$ to some
\[
j^*: V[H][E] \to N[H^*][E^*]
\]
which is defined in $V[H][E][F][H']=V[H][E][H'][F].$ Define a narrow system $\mathcal{S} \in V[H][E]$ as in the proof of Lemma 3.17.
Applying the preservation Lemma 2.4, we argue that $\mathcal{S}$ has a cofinal branch in $V[H][E][F].$ So in
$V[H][E][F]$ we can find  an unbounded
$J \subseteq I, \alpha \mapsto \xi_\alpha$ and $\Lambda = \langle \langle \bar{g}, \bar{f}, \bar{H}, \bar{F}  \rangle, \bar{d} \rangle \in \Xi_{\gamma_0+1}$ with
 $\max(\dom(g_\Lambda))=\gamma_0$ such that for all
$\alpha < \beta$ in $J$ we have $\Lambda \Vdash^*$`` $\langle \alpha, \xi_\alpha \rangle <_{\lusim{T}}      \langle \beta, \xi_\beta  \rangle$''.
Then by the argument in Lemma 3.2 from \cite{neeman0}, we can get that $J, \Lambda$ and $\alpha \mapsto \xi_\alpha$
in $V[H][E].$ Setting $u_\alpha= \langle \alpha, \xi_\alpha  \rangle$ for $\alpha \in J$ we get that for
$\alpha < \beta$ in $J$, there is a condition $r \in \MRB$ which satisfies the required properties of the lemma.
\end{proof}
Fix $J \subseteq \kappa_\eta^+,$  $\langle \langle \bar{g}, \bar{f}, \bar{H}, \bar{F}  \rangle, \bar{d} \rangle$
and the sequence of nodes $\langle u_\alpha: \alpha \in J \rangle$ as in the conclusion of the above lemma.
By shrinking $J$, if necessary, we can assume that for some $\epsilon < \kappa,$ each $u_\alpha= \langle \alpha, \epsilon  \rangle$.
The next lemma corresponds to Lemma 3.19 of last section.
\begin{lemma}
Suppose that $\Lambda_0=\langle  \langle g, f, H, F \rangle, d \rangle \in \Xi_{\gamma_0}$, $L \subseteq \kappa_\eta^+$ is unbounded and for all
$\alpha < \beta$ in $L$, we have $\Lambda_0 \Vdash^*$``$u_\alpha <_{\lusim{T}} u_\beta$''. Let $\gamma_0 \leq \gamma < \eta.$ Then there are $\rho< \kappa_\eta^+$
and $\langle \Lambda_\alpha: \alpha \in L \setminus \rho       \rangle \in V[H][E]$ such that:
\begin{enumerate}
\item $\Lambda_\alpha \in \Xi_{\gamma+1}, stem(\Lambda_\alpha)=stem(\Lambda_0)$ and $\Lambda_\alpha \upharpoonright \gamma_0=\Lambda_0,$


\item For all $\alpha < \beta$ in $L \setminus \rho$ and for all $y \in H_{\Lambda_\alpha}(\gamma) \cap H_{\Lambda_\beta}(\gamma)$ such that  $F_{\Lambda_\alpha}(\gamma)(y)$ and $F_{\Lambda_\beta}(\gamma)(y)$
are compatible, we have
\[
 \langle \langle  g ^{\frown} \langle \gamma, y \rangle , f ^{\frown} \langle \gamma, F_{\Lambda_\alpha}(\gamma)(y) \cup F_{\Lambda_\beta}(\gamma)(y) \rangle, H_{\Lambda_\alpha} \wedge H_{\Lambda_\beta} \upharpoonright \gamma, F_{\Lambda_\alpha} \vee F_{\Lambda_\beta} \upharpoonright \gamma \rangle,  d                          \rangle \Vdash^* u_\alpha <_{\lusim{T}} u_\beta.
\]
\end{enumerate}
\end{lemma}
\begin{proof}
We follow the proof of Lemma 3.19 and modify it using ideas of the proof of Lemma 16 from \cite{sinapova3} to get the result. We present the proof in some details, as it requires major
modifications with respect to Lemma 3.19.

Let $j: V \rightarrow N$ be a $\kappa_\eta^+$-supercompact embedding with critical point $\kappa_{\gamma+4}.$ As in the previous lemma, extend $j$ to some
\[
j_1: V[H] \rightarrow N[H^*],
\]
where $H^*$ is $j(\MCB)$-generic over $V$. We can arrange so that $H^*=H*E*H'$, where $H'$
is $\MCB'$-generic over $V[H][E],$ for some $\MCB'$ which is
 $\kappa_{\gamma+4}$-distributive in $V[H][E]$ and it is
$\kappa_{\gamma+4}$-closed in $V[H].$ Let $F$ be $\Add(\kappa, j(\kappa_\eta^{++}))$-generic over $V[H][E].$
Then $F$ is also generic over $V[H][E][H']$ and so
\[
E^*=\{f \in \Add(\kappa, j(\kappa_\eta^{++})): f \upharpoonright j''\kappa_\eta^{++} \in j_1''E, f \upharpoonright (j(\kappa_\eta^{++}) \setminus j''\kappa_\eta^{++}) \in F           \}.
\]
 is $\Add(\kappa, j(\kappa_\eta^{++}))$-generic over $V[H^*]=V[H][E][H']$ such that
$j_1''E \subseteq E^*$. It follows that we can extend $j$ to some
\[
j^*: V[H][E] \to N[H^*][E^*]
\]
which is defined in $V[H][E][F][H']=V[H][E][H'][F].$

First we show that there are $\rho< \kappa_\eta^+$ and $\langle \Lambda^*_\alpha: \alpha \in L \setminus \rho       \rangle$
as in the statement of the lemma in $V[H][E][H'].$ Then we will use the preservation Lemma 2.4 to show that we can find these objects in $V[H][E].$

For $\alpha < \beta$ in $L$,  by our assumption of the lemma, $\Lambda_0 \Vdash^*$``$u_\alpha <_{\lusim{T}} u_\beta$'', and so we can find
$r \in \MRB$ such that
\begin{itemize}
\item $stem(r)= \langle g, f, , d \rangle,$
\item $r \upharpoonright \gamma = \Lambda_0,$
\item $r\Vdash$``$u_\alpha <_{\lusim{T}} u_\beta$''.
\end{itemize}
As $K_\gamma$ is closed under $\kappa_\eta^+$-sequences, so we can assume that $[F^r(\gamma)]_{U_\gamma}$
is the same for all $r$'s as above. We also assume that for all such $r$'s, $\dom(F^r)(\gamma)=H^r(\gamma)=B_\gamma.$

Now let $\mu \in j^*(L)$ be such that $\sup j''\kappa_\eta^+ < \mu < j(\kappa_\eta^+).$
Note that $j^*(\Lambda_0)= \Lambda_0$. By elementarity, for each $\alpha \in L$, we can find $r_\alpha = (p_\alpha, d) \in j^*(\MRB)$
with $stem(r_\alpha)= \langle g, f, d \rangle$ and $r_\alpha \upharpoonright \gamma_0 = \Lambda_0$ such that
\[
r_\alpha \Vdash~\langle   j(\alpha), \epsilon \rangle <_{j^*(\lusim{T})} \langle \mu, \epsilon  \rangle.
\]
For each $\alpha \in L$, denote $r_\alpha = \langle  \langle g, f, H_\alpha, F_\alpha     \rangle      ,d  \rangle$. We may assume that for each $\alpha, [F_\alpha(\gamma)]_{j^*(U_\gamma)}= [F^*]_{j^*(U_\gamma)}$, for some fixed $F^*$.

Note that in this case, the measures $j^*(U_\gamma)$ and  $U_\gamma$  are different, so we can not argue as in the proof of Lemma 3.19. We use ideas of Neeman \cite{neeman0} to overcome this difficulty.

For each $y \in B_\gamma$ set
\[
J_y=\{ \alpha \in L:  \langle \langle  g ^{\frown} \langle \gamma, y \rangle , f ^{\frown} \langle \gamma, F_{\alpha}(\gamma)(y) \rangle, H_{\alpha} \upharpoonright \gamma, F_{\alpha}  \upharpoonright \gamma \rangle,  d                          \rangle \Vdash^* \langle   j(\alpha), \epsilon \rangle <_{j^*(\lusim{T})} \langle \mu, \epsilon  \rangle \}.
\]
Also let
\begin{center}
$W_y=\{ C \subseteq L: C$ is unbounded and $\exists b \in \Add(\kappa, j(\kappa_\eta^{++}))$ such that $ b \Vdash \lusim{J}_y = C        \}$.
\end{center}
Then each $W_y$ and $y \mapsto W_y$ is in $V[H][E][H']$ and as in \cite{neeman0}, we have that
\begin{itemize}
\item If $J_y$ is unbounded in $\kappa_\eta^+,$ then $J_y \in V[H][E][H']$,
\item If $C_1 \neq C_2$ are both in $W_y$, then they are disjoint on a tail end,
\item For all $C \in W_y,$ if $\alpha < \beta$ in $C$, then
\[
 \langle \langle  g ^{\frown} \langle \gamma, y \rangle , f ^{\frown} \langle \gamma, F_{\alpha}(\gamma)(y) \rangle, H_{\alpha} \upharpoonright \gamma, F_{\alpha}  \upharpoonright \gamma \rangle,  d                          \rangle \Vdash^* \langle   j(\alpha), \epsilon \rangle <_{j^*(\lusim{T})} \langle \mu, \epsilon  \rangle.
\]
\end{itemize}
Let $\rho < \kappa_\eta^+$ be such that for all $C_1, C_2 \in W_y, C_1$ and $C_2$ are disjoint above $\rho.$
For $\alpha \in L \setminus\rho$ and $y\in B_\gamma$ define $h(y, \alpha)$ to be the unique $C \in W_y$
such that $\alpha \in C$, if such a $C$ exists and undefined otherwise. Note that if $J_y$
is unbounded in $\kappa_\eta^+$ and $\alpha \in J_y,$ then $h(y, \alpha)=J_y$. Let $\alpha_0=\min(L \setminus \rho)$
and for $\alpha \in L \setminus\rho$ let
\[
A^*_\alpha = \{ y \in B_\gamma: h(y, \alpha) =h(y, \alpha_0)   \}.
\]
By the arguments in \cite{neeman0}, each $A^*_\alpha \in U_\gamma.$ For each $\alpha \in L \setminus \rho$ set $\Lambda^*_\alpha \in \Xi_{\gamma+1}$ be such that
\begin{itemize}
\item $stem(\Lambda^*_\alpha)=stem(\Lambda_0)$,
\item  $\Lambda^*_\alpha \upharpoonright \gamma_0=\Lambda_0,$
\item $H_{\Lambda^*_\alpha}(\gamma)= A^*_\alpha$,
\item $F_{\Lambda^*_\alpha}(\gamma)= F^* \upharpoonright A^*_\alpha.$
\end{itemize}
Then $\rho < \kappa_\eta^+$ and $\langle \Lambda^*_\alpha: \alpha \in L \setminus \rho       \rangle$ are as in the statement of the lemma, but they are in $V[H][E][H']$
and not necessarily in $V[H][E]$. We now use the preservation Lemma 2.4 to show that there are such objects in $V[H][E]$.

Let $\Theta= \{ \Lambda \in \Xi_{\gamma+1}: \Lambda \upharpoonright \gamma_0 =\Lambda_0$,  $stem(\Lambda)=stem(\Lambda_0)$ and $F_\Lambda(\gamma)=F^* \upharpoonright H_\Lambda(\gamma)   \}$.
For $\Lambda \in \Theta$  and $y \in H_\Lambda(\gamma),$ we define $\Pi=\Lambda^{\frown} \langle y \rangle \in  \Xi_{\gamma+1}$ be such that
$g_\Pi= g_\Lambda^{\frown} \langle \gamma, y   \rangle$, $f_\Pi= f^{\frown} \langle \gamma, F_\Lambda(\gamma)(y)      \rangle$,
$d_\Pi=d_\Lambda$ and we set $F_\Pi, H_\Pi$ be the same as $F_\Lambda, H_\Lambda$ with $\gamma$ removed from their domain.

For $\Lambda \in \Theta$ and $y \in H_\Lambda(\gamma)$ set
\[
b_{\Lambda, y} = \{ \alpha \in L \setminus \rho:    y \in   A^*_\alpha            \}.
\]
Note that $b_{\Lambda, y} \in V[H][E][H']$ and for all $\alpha < \beta$ in $b_{\Lambda, y}$ we have $\Lambda^{\frown} \langle y \rangle \Vdash^* \langle \alpha, \epsilon \rangle <_{\lusim{T}} \langle \beta, \epsilon \rangle$.
\begin{claim}
For every $\Lambda \in \Theta$ and $A\subseteq H_\Lambda(\gamma), A \in U_\gamma,$  there is $y \in A$ such that  $b_{\Lambda, y}$
is unbounded and there is an unbounded set $b \subseteq \kappa_\eta^+$ in $V[H][E]$ such that $b_{\Lambda, y} \subseteq b$ and for all $\alpha < \beta$ in $b$,
$\Lambda^{\frown} \langle y \rangle \Vdash^* \langle \alpha, \epsilon \rangle <_{\lusim{T}} \langle \beta, \epsilon \rangle$.
\end{claim}
\begin{proof}
Consider the narrow system $\mathcal{S}= \langle L \setminus \rho, \mathcal{R}  \rangle$,
where
\[
\mathcal{R} = \{ R_{\Lambda, y}: \Lambda \in \Theta, y \in H_\Lambda(\gamma)             \}
\]
and for nodes $a, b$ we have
\[
a R_{\Lambda, y} b \iff \Lambda^{\frown}\langle y \rangle \Vdash^* a <_{\lusim{T}} b.
\]
Also consider the branches $\{b^*_{\Lambda, y}: \Lambda \in \Theta, y \in H_\Lambda(\gamma)     \}$, where
$b^*_{\Lambda, y} = \{ \langle \alpha, \epsilon \rangle:    \alpha \in b_{\Lambda, y}            \}$
if $y \in A$, and $b^*_{\Lambda, y}=\emptyset$ otherwise. By preservation Lemma 2.4 (see also Remark 12 in \cite{sinapova3}), in $V[H][E]$, we can find an unbounded $b \subseteq \Lambda$
and $y \in H_{\Lambda}(\gamma)$ such that $b_{\Lambda, y} \subseteq b$ and $b_{\Lambda, y}$ is unbounded. Since $b_{\Lambda, y}$
is unbounded, we have $y \in A$. Then $b$ is as required.
\end{proof}
For $\Lambda \in \Theta$ and $y \in H_\Lambda(\gamma)$ let $\dagger_{\Lambda, y}$ be the assertion: $b_{\Lambda, y}$
is unbounded and there is an unbounded $b \subseteq \kappa_\eta^+$ in $V[H][E]$ with $b_{\Lambda, y} \subseteq b$
such that for all $\alpha < \beta$ in $b$, $\Lambda^{\frown} \langle y \rangle \Vdash^* \langle \alpha, \epsilon \rangle <_{\lusim{T}} \langle \beta, \epsilon \rangle$.
It follows from the above claim that, for each $\Lambda \in \Theta,$
the set
\[
A_\Lambda= \{ y \in H_\Lambda(\gamma): \dagger_{\Lambda, y} \text{~holds~}    \}
\]
is in $V[H][E]$ and $A_\Lambda \in U_\gamma$ ($A_\Lambda$ is in $V[H][E]$ since $\MCB'$ is enough distributive. If $A_\Lambda$ is not in $U_\gamma,$
then its complement, i.e., $Y=\{ y  \in H_\Lambda(\gamma):  \dagger_{\Lambda, y}$ fails $  \}$, is in $U_\gamma.$ Apply the above claim to $Y$ to get a contradiction).

For each $\Lambda \in \Theta$ and $y \in A_\Lambda$ let $L_{\Lambda, y} \subseteq L$ witness
 $\dagger_{\Lambda, y}$. Note that by the distributivity of $\MCB'$, we have that $\langle \Lambda, y  \rangle \mapsto L_{\Lambda, y}$
is in $V[H][E].$ For $\Lambda \in \Theta$ and $\alpha \in L \setminus \rho$  define
\[
A_{\Lambda, \alpha} = \{ y \in H_\Lambda(\gamma) : \alpha \in   L_{\Lambda, y}  \}.
\]
Then as in Claim 20 in \cite{sinapova3}, each $A_{\Lambda,\alpha} \in U_\gamma$.
The rest of the argument is as in Lemma 3.19
and we are done.
\end{proof}
The next lemma is analogue to Lemma 3.20 and whose proof is essentially the same.
\begin{lemma}
There are $\rho< \kappa_\eta^+$ and conditions $\langle p_\alpha: \alpha \in J \setminus \rho     \rangle$
with $p_\alpha \upharpoonright \gamma_0+1 = \langle \langle \bar{g}, \bar{f}, \bar{H}, \bar{F}  \rangle, \bar{d} \rangle$
such that for all $\alpha < \beta$ in $J \setminus \rho, p_\alpha \wedge p_\beta \Vdash$``$u_\alpha <_{\lusim{T}} u_\beta$''.
Here $ p_\alpha \wedge p_\beta $ denotes the weakest extension of $p_\alpha$ and $p_\beta.$
\end{lemma}
The rest of the argument is as before. This completes our sketch of the proof of Theorem 1.2.

\section{Product of Levy collapses}
In this section we present another proof of Theorem 1.1, using  ideas of Neeman \cite{neeman}. The idea of using Neeman's method to prove the theorem was suggested by Yair Hayut \cite{hayut}.  The proof is much easier than the proof given in the previous section.
However let us mention that the method given in this section can not be used to Theorem 1.2.
 Let $\langle  \kappa_\xi: \xi \leq \eta      \rangle$
be an increasing and continuous sequence of cardinals so that:
\begin{enumerate}
\item Each $\kappa_\xi,$ where $\xi=0$ or $\xi$ is a successor ordinal is a supercompact cardinal,
\item $\kappa_\eta=\lambda.$
\item If $\xi=0$ or $\xi$ is a successor ordinal, then  $\kappa_\xi$ is Laver indestructible under $\kappa_\xi$-directed closed forcing notions,
\item For all $\xi \leq \eta, 2^{\kappa_\xi}=\kappa_{\xi}^+.$
\end{enumerate}
As before we can assume that
$\eta<\kappa_0$.

Let $\mathbb{C}$ be the Easton support iteration of $\Col((\kappa_i)_*^{++}, < \kappa_{i+1}), i<\eta,$ and let $H$ be $\mathbb{C}$-generic over $V$. The following lemma can be proved easily:
\begin{lemma}
$(a)$ $\mathbb{C}$ is $\kappa$-directed closed, in particular $\kappa$ remains supercompact in $V[H].$

$(b)$ In $V[H]$, the only measurable cardinals in the interval $(\kappa, \kappa_\eta)$ are $(\kappa_i)_*, i<\eta$ and $\kappa_\eta$ is the limit of them; in particular,
$\kappa_\eta$ is the limit of the first $\eta$ measurable cardinals above $\kappa.$

$(c)$ $NSP(\kappa_\eta^+)$ holds in $V[H],$ where $NSP(\kappa_\eta^+)$ is the assertion  ``the Narrow System Property holds at $\kappa_\eta^+$''.
\end{lemma}
\begin{proof}
$(a)$ and $(b)$ are clear; for $(c),$ see \cite{LambieHanson2015}.
\end{proof}
The following lemma completes the proof.
\begin{lemma}
Work in $V[H]$. There is $\mu < \kappa =\kappa_0$ such that in the generic extension of $V[H]$
by $\Col(\omega, \mu) \times \Col(\mu^{+}, \kappa)$ the tree property at $\kappa_{\eta}^+$ holds.
\end{lemma}
\begin{proof}
The proof of the lemma, which in conjunction with $NSP$ gives the tree property,  follows arguments from Neeman's paper \cite{neeman}, adapting them for the uncountable cofinality case. We may also mention that the supercompactness of the other cardinals is used in the proof of $NSP$, see \cite{LambieHanson2015}.

Assume that for all $\mu < \kappa$ the forcing $\Col(\omega,\mu)\times\Col(\mu^{+},\kappa)$ adds an Aronszajn tree at $\kappa_\eta^+$ and let $\lusim{T}_\mu$ be a name for this tree.

Let $j\colon V[H] \to M$ be a $\kappa_\eta^+$-supercompact elementary embedding, i.e., $\crit(j) = \kappa$ and $M$ is closed under $\kappa_\eta^+$-sequences.
Let $\vec{T}=\langle \lusim{T}_\mu: \mu<\kappa \rangle$ and set
$j(\vec{T})= \langle \lusim{T}^*_\mu: \mu<j(\kappa) \rangle.$

In $M$, let us pick $\mu = \kappa_\eta < j(\kappa)$, so $\lusim{T}^*_{\kappa_\eta}$ is a $\Col(\omega, \kappa_\eta)\times \Col(\kappa_\eta^{+},j(\kappa))$-name for an Aronszajn tree at $j(\kappa_\eta^+)$. Let $\delta = \sup j''(\kappa_\eta^+) < j(\kappa_\eta^+)$. By recursion on $\alpha < \kappa_\eta^+$, and using the
$\kappa_\eta^{+}$-closure of $\Col^M(\kappa_\eta^{+},j(\kappa))$,
 pick a sequence  $\langle (p_\alpha, q_\alpha): \alpha < \kappa_\eta^+ \rangle$ such that:
\begin{enumerate}
\item $(p_\alpha, q_\alpha)\in\Col(\omega,\kappa_\eta) \times \Col^M(\kappa_\eta^+,j(\kappa))),$

\item $\langle q_\alpha: \alpha < \kappa_\eta^+ \rangle$ is a decreasing sequence in $\Col^M(\kappa_\eta^+,j(\kappa)))$,

\item There is an ordinal $\zeta_\alpha < j(\kappa_\eta)$ such that $(p_\alpha,q_\alpha)\Vdash \langle j(\alpha),\zeta_\alpha\rangle \leq_{\lusim{T}^*_{\kappa_\eta}} \langle \delta, 0\rangle$.

\end{enumerate}
Let $I\subseteq \kappa_\eta^+$ be a cofinal set such that for all $\alpha\in I$, $p_\alpha = p_\star$ and $\zeta_\alpha < j(\kappa_{\sigma})$ for some fixed $p_\star\in\Col(\omega,\kappa_\eta)$ and $\sigma < \eta$.

For every $\alpha, \beta\in I$, $\alpha < \beta$, $M$ satisfies that there is a condition $r = (p_\star, q_\beta)$ and ordinals $\zeta_\alpha, \zeta_\beta < j(\kappa_\sigma)$ such that that $r$ forces $\langle j(\alpha), \zeta_\alpha\rangle \leq_{\lusim{T}^*_{\kappa_\eta}} \langle j(\beta), \zeta_\beta\rangle$. Reflecting this downwards (for every pair $\alpha, \beta$ separately), in $V[H]$ there is $\rho < \kappa$, a condition $r$ and $\zeta, \eta < \kappa_\sigma$ such that $r\Vdash \langle \alpha, \zeta\rangle \leq_{\lusim{T}_\rho} \langle \beta, \eta\rangle$. This defines a narrow system on $I$ of width $\kappa_\sigma$ and $\kappa$ relations. Namely, for a pair of ordinal $\rho < \kappa$ and condition $r\in\Col(\omega,\rho)\times\Col(\rho^+,\kappa)$ we say that $\langle \alpha, \zeta\rangle \leq_{\langle \rho, r\rangle} \langle \beta, \eta\rangle$ iff $r\Vdash \langle \alpha, \zeta\rangle \leq_{\lusim{T}_\rho} \langle \beta, \eta\rangle$.

A branch in this narrow system consists of a single condition that forces a branch through a name of an Aronszajn tree. Since the narrow system property holds in $V[H]$, such a branch exists so we get a contradiction.
\end{proof}
\section{Tree property at successor of arbitrary singular cardinals}
In this section we use the method of Section 5   to prove Theorem 1.3 which gives us tree property at successor of an arbitrary singular cardinal.
 Let $\langle  \kappa_\xi: \xi \leq \eta      \rangle$
be as in Section 5 with $\eta < \kappa_0.$
Let $\mathbb{C}$ be the Easton support iteration of $\Col(\kappa_i^{+}, < \kappa_{i+1}), i<\eta,$ and let $H$ be $\mathbb{C}$-generic over $V$.
Note that $V[H]\models~\kappa_\eta=\kappa_0^{+\eta}$. The next lemma can be proved as in Lemma 4.2.
\begin{lemma}
Work in $V[H]$. There is $\mu < \kappa$ such that in the generic extension of $V[H]$
by $\Col(\omega, \mu) \times \Col(\mu^{+}, \kappa)$ the tree property at $\kappa_{\eta}^+$ holds.
\end{lemma}
Let $G_1\times G_2$ be $\Col(\omega, \mu) \times \Col(\mu^{+}, \kappa)$-generic over $V[H]$
such that $V[H][G_1 \times G_2] \models$``The tree property holds at $\kappa_\eta^+$''.
But  $\aleph_{\eta+1}^{V[H][G_1 \times G_2]}=\kappa_\eta^+$,
and the result follows.

\subsection*{Acknowledgements}
The author thanks the referees of the paper for many helpful comments and corrections.
He also thanks Yair Hayut  for his helps for the results of section 5; in particular the use of Neeman's idea  was suggested by him.

School of Mathematics, Institute for Research in Fundamental Sciences (IPM), P.O. Box:
19395-5746, Tehran-Iran.

E-mail address: golshani.m@gmail.com



\end{document}